\documentclass[11pt]{article}

\usepackage[a4paper,margin=1in]{geometry}
\usepackage{amsmath,amssymb,amsthm,mathtools,bm}
\usepackage{booktabs}
\usepackage{graphicx}
\usepackage{enumitem}
\usepackage{hyperref}
\usepackage{booktabs}
\usepackage{tabularx}
\usepackage{array}

\newtheorem{theorem}{Theorem}[section]
\newtheorem{lemma}[theorem]{Lemma}
\newtheorem{corollary}[theorem]{Corollary}
\newtheorem{proposition}[theorem]{Proposition}
\theoremstyle{definition}
\newtheorem{definition}[theorem]{Definition}
\newtheorem{assumption}[theorem]{Assumption}
\newtheorem{example}[theorem]{Example}
\newtheorem{remark}[theorem]{Remark}

\newcommand{\R}{\mathbb{R}}
\newcommand{\E}{\mathbb{E}}
\newcommand{\Var}{\operatorname{Var}}
\newcommand{\Bias}{\operatorname{Bias}}
\newcommand{\MISE}{\operatorname{MISE}}
\newcommand{\AMISE}{\operatorname{AMISE}}
\newcommand{\supp}{\operatorname{supp}}
\newcommand{\co}{\operatorname{co}}
\newcommand{\esssup}{\operatorname*{ess\,sup}}
\newcommand{\norm}[1]{\left\lVert #1\right\rVert}
\newcommand{\abs}[1]{\left\lvert #1\right\rvert}

\newcommand{\tr}{\operatorname{tr}}

\title{Weak-Curvature AMISE and Plug-in Bandwidth Selection for Kernel Density Estimation}
\author{Alireza Kabgani\thanks{Corresponding author. Department of Mathematics, University of Antwerp, Antwerp, Belgium. Email: alireza.kabgani@uantwerp.be}
\and Elaheh Lotfian\thanks{Antwerp, Belgium. Email: e.lotfian@gmail.com} }

\date{}

\begin{document}
\maketitle

\begin{abstract}
Kernel density estimation risk expansions are commonly expressed through the integrated squared curvature term that enters second-order AMISE and plug-in bandwidth rules. This paper develops a weak-curvature formulation of this classical calculation for densities whose second derivative exists weakly rather than as a continuous classical function. We prove that if a density has square-integrable weak curvature, then the standard second-order AMISE expansion, oracle bandwidth order, and kernel-dependent optimality calculation remain valid with the curvature functional understood in the weak sense. The class $C^{1,1}(\mathbb{R})\setminus C^2(\mathbb{R})$ serves as a concrete and practically relevant subclass: the first derivative is Lipschitz, while curvature may be kinked, discontinuous, or undefined at isolated points. Building on this formulation, we introduce a generalized-curvature plug-in (GCPI) bandwidth selector. The selector estimates the weak-curvature functional by a pilot density-derivative estimator with a leave-one-out U-statistic correction and substitutes this estimate into the AMISE bandwidth formula. We prove first-order oracle equivalence under ratio-consistent weak-curvature estimation and establish consistency of the proposed U-statistic curvature estimator under explicit pilot-bandwidth conditions. We also give a scalar-bandwidth multivariate extension based on weak Hessians and illustrate the theory through nonsmooth density examples, simulations, and a real-data application.
\end{abstract}

\noindent\textbf{Keywords:} Kernel density estimation, nonparametric estimation, $C^{1,1}$ smoothness, AMISE, plug-in bandwidth selection

\noindent\textbf{MSC 2020:} Primary 62G07, 62G20; secondary 49J52, 26A16.


\section{Introduction}
\label{sec:introduction}

Let $X_1,\ldots,X_n$ be independent observations in $\mathbb{R}^d$ with common density $f:\mathbb{R}^d\to[0,\infty)$. Kernel Density Estimation (KDE) is a standard nonparametric method for estimating $f$; see, for example, \cite{Parzen1962,Rosenblatt1956,Scott2015,Silverman1986,WandJones1995}. For a scalar bandwidth $h>0$ and a kernel $K:\mathbb{R}^d\to\mathbb{R}$, the multivariate kernel density estimator is
\[
\widehat f_h(x) := \frac{1}{nh^d} \sum_{i=1}^n K\left(\frac{x-X_i}{h}\right), \qquad x\in\mathbb{R}^d.
\]
Many textbook derivations of the bias, Mean Integrated Squared Error (MISE), and Asymptotic Mean Integrated Squared Error (AMISE) use a pointwise second-order Taylor expansion and therefore assume that the unknown density is twice continuously differentiable. Under this convenient assumption, one obtains the usual leading bias term, the classical AMISE formula, the bandwidth rate $n^{-1/5}$ in one dimension, and the scalar-bandwidth rate $n^{-1/(d+4)}$ in $d$ dimensions. The Epanechnikov kernel then arises as the classical optimal nonnegative second-order kernel in the AMISE sense; see \cite{Epanechnikov1969,Scott2015,Silverman1986,Tsybakov2009,WandJones1995}.

It is important to distinguish this textbook route from the broader theory of density estimation under weak smoothness. Kernel, wavelet, and adaptive density estimation over H\"older, Sobolev, and Besov-type classes are classical; see, for example, \cite{DonohoJohnstoneKerkyacharianPicard1996,GineNickl2016,GoldenshlugerLepski2011,GoldenshlugerLepski2014,KerkyacharianPicard1993,KerkyacharianPicardTribouley1996,Tsybakov2009}. In such settings, one can obtain bias and risk bounds under assumptions much weaker than pointwise $C^2$ smoothness. Thus the present paper does not claim that $C^{1,1}$ is minimal for KDE consistency, nor that the rate $n^{-4/5}$ is new in a weak-smoothness framework.

Our objective is more specific. We revisit the classical second-order AMISE formula and the associated plug-in bandwidth calculation in a setting where the curvature exists naturally as a weak derivative. In one dimension, the relevant functional is
\[
R(f'') := \int_{\mathbb{R}}\{f''(x)\}^2\,dx,
\]
where $f''$ is interpreted as the weak second derivative. The concrete regularity class emphasized in this paper is $C^{1,1}(\mathbb{R})\setminus C^2(\mathbb{R})$: the first derivative is Lipschitz continuous, so a weak second derivative exists and is locally bounded, while the classical second derivative may fail to exist at some points or may be discontinuous. This class includes natural kinked-curvature densities arising from threshold or regime-change models, heterogeneous populations, robust mixture specifications, and boundary-affected density estimation; see, for example, \cite{ClineHart1991,Hansen2017,Jones1993,McLachlanPeel2000}.

The first contribution is a weak-curvature derivation of the classical AMISE expansion. The usual pointwise $C^2$ Taylor expansion is replaced by a weak integral Taylor argument and $L^2$ translation continuity. Under the Sobolev-type condition $f\in L^2(\mathbb{R})$ and $f''\in L^2(\mathbb{R})$ in the weak sense, we prove that the classical AMISE expansion remains valid with $R(f'')$ understood as weak curvature. Consequently, the standard bandwidth order $n^{-1/5}$, the MISE rate $n^{-4/5}$, and the classical kernel-dependent optimization argument are retained without assuming a continuous classical second derivative. The $C^{1,1}\setminus C^2$ setting is used to give a concrete nonsmooth class and a local curvature interpretation.

The second and more methodological contribution is a generalized-curvature plug-in (GCPI) bandwidth selector. Classical plug-in bandwidth selectors replace the unknown curvature quantity in the AMISE formula by an estimate; see \cite{Chiu1996,HallSheatherJonesMarron1991,Heidenreich2013,JonesMarronSheather1996,ParkMarron1990,Sheather2004,SheatherJones1991}. In the present formulation, the target is explicitly the weak-curvature functional $R(f'')$. We estimate it using a pilot density derivative estimator and a leave-one-out U-statistic correction. We prove that any ratio-consistent estimator of $R(f'')$ yields a plug-in bandwidth that is first-order AMISE-equivalent to the oracle bandwidth, and we establish consistency of the proposed U-statistic curvature estimator under explicit pilot-bandwidth conditions. This connects the present work to the literature on integrated squared density derivative estimation, including \cite{BickelRitov1988,HallMarron1987,JonesSheather1991,VanEsHoogstrate1994,Wu1995}.

There is also a line of work on nonsmooth MISE and integrated squared error expansions, including settings where discontinuities or jumps in higher-order derivatives affect the leading terms; see, for example, \cite{VanEeden1985,VanEsHoogstrate1994,VanEs1997}. The present paper is complementary to that literature. Its focus is not on a new discontinuity correction or a new minimax theory, but on showing that the ordinary second-order AMISE and plug-in bandwidth logic can be written cleanly in terms of weak curvature for densities in $C^{1,1}\setminus C^2$.

The multivariate extension is developed for scalar bandwidths. In this case, the weak second derivative is replaced by the weak Hessian $D^2f$, and the leading curvature term is expressed through
\[
\mathcal L_K f(x) := \tr\{M_2(K)D^2f(x)\}, \qquad M_2(K) := \int_{\mathbb{R}^d}uu^\top K(u)\,du.
\]
Under Lipschitz-gradient regularity and $D^2f\in L^2(\mathbb{R}^d)$ in the weak sense, we obtain the scalar-bandwidth multivariate AMISE expansion and recover the usual rate $n^{-4/(d+4)}$. Full bandwidth-matrix theory, anisotropic adaptation, and multivariate plug-in bandwidth selection are left for future work.

The main contributions are as follows.
\begin{itemize}
    \item We give a self-contained weak-curvature derivation of the classical second-order AMISE expansion under the Sobolev-type condition $f\in L^2(\mathbb{R})$ and $f''\in L^2(\mathbb{R})$ in the weak sense.
    \item We clarify the distinction between pointwise $C^2$ smoothness, $C^{1,1}$ regularity, and weak $L^2$ curvature, and use $C^{1,1}\setminus C^2$ densities as concrete examples where the weak-curvature interpretation is meaningful.
    \item We introduce a generalized-curvature plug-in bandwidth selector and prove first-order oracle equivalence under ratio-consistent estimation of $R(f'')$.
    \item We prove consistency of a leave-one-out U-statistic estimator of the weak-curvature functional under explicit pilot-bandwidth conditions.
    \item We extend the weak-curvature AMISE argument to multivariate densities with Lipschitz gradients and square-integrable weak Hessians.
\end{itemize}

The paper also gives examples of densities in $C^{1,1}\setminus C^2$, numerical experiments, and a real-data illustration. These examples are included to clarify the scope of the theory and to show how the weak-curvature plug-in rule behaves beyond the formal asymptotic setting.

The remainder of the paper is organized as follows. Section~\ref{sec:preliminaries} introduces the notation, kernel assumptions, weak derivatives, and the $C^{1,1}$ framework. Section~\ref{sec:bias-mise} proves the basic bias and MISE bounds. Section~\ref{sec:amise} derives the weak-curvature AMISE expansion. Section~\ref{sec:gcpi} develops the GCPI bandwidth selector and the U-statistic curvature estimator. Section~\ref{sec:multivariate} gives the multivariate scalar-bandwidth extension. Section~\ref{sec:kernel-optimality} discusses kernel optimality. Section~\ref{sec:generalized-second-order} gives a local generalized second-order interpretation. Section~\ref{sec:examples} presents examples, Section~\ref{sec:numerics} gives numerical illustrations, and Section~\ref{sec:discussion} concludes.

\section{Preliminaries}\label{sec:preliminaries}

We begin by fixing the notation and assumptions used in the subsequent analysis. In particular, we introduce the kernel conditions, the $C^{1,1}$ regularity, and weak second derivatives. 
Unless otherwise stated, Sections~\ref{sec:preliminaries}--\ref{sec:gcpi} are formulated in the univariate setting. The multivariate notation is introduced separately in Section~\ref{sec:multivariate}. 

Throughout the paper, all integrals are understood in the Lebesgue sense, and functions that agree almost everywhere are identified.
For $1\le p<\infty$, let $L^p(\mathbb{R})$ denote the space of measurable functions $g:\mathbb{R}\to\mathbb{R}$ such that
\[
\|g\|_p :=\left(\int_{\mathbb{R}} |g(x)|^p\,dx\right)^{1/p}<\infty.
\]
We also write $\|g\|_{\infty} :=\esssup_{x\in\mathbb{R}} |g(x)|$,
when $g$ is essentially bounded. For $g\in L^2(\mathbb{R})$, define
\[
R(g) :=\int_{\mathbb{R}} g(x)^2\,dx=\|g\|_2^2.
\]
In particular, $R(f'')$ denotes the weak-curvature functional when $f''$ is interpreted as a weak second derivative (see Subsection~\ref{subsec:c11}).
For a measurable function $g:\mathbb R\to\mathbb R$, we write
\[
 \supp g:=\overline{\{x\in\mathbb R:\; g(x)\ne0\}},
\]
where the closure is taken in $\mathbb R$.

\subsection{Kernel assumptions}

For a kernel $K$, define its $j$th moment by
\[
\mu_j(K):=\int_{\mathbb{R}} u^j K(u)\,du,
\]
whenever the integral is absolutely finite, that is,
\[
\int_{\mathbb{R}} |u|^j |K(u)|\,du <\infty.
\]

We impose the following standing assumptions.
\begin{assumption}[\textbf{Kernel conditions}]\label{ass:kernel}
The kernel $K:\mathbb{R}\to\mathbb{R}$ satisfies
\begin{enumerate}[label=(K\arabic*)]
    \item\label{ass:kernel:1} $K\in L^1(\mathbb{R})\cap L^2(\mathbb{R})$;
    \item\label{ass:kernel:2} $\int_{\mathbb{R}}K(u)\,du=1$;
    \item\label{ass:kernel:3} $\int_{\mathbb{R}}uK(u)\,du=0$;
    \item\label{ass:kernel:4} $\int_{\mathbb{R}}u^2|K(u)|\,du<\infty$.
\end{enumerate}
No nonnegativity assumption is imposed on $K$ unless stated explicitly. Thus the assumptions allow signed kernels, although nonnegative kernels are considered when kernel optimality is discussed. If $K$ is symmetric, then \ref{ass:kernel:3} follows automatically whenever $\int_{\mathbb{R}} |uK(u)|\,du<\infty$. This integrability condition is implied by \ref{ass:kernel:1} and \ref{ass:kernel:4}.
\end{assumption}

For $h>0$, define the rescaled kernel $K_h(x) :=h^{-1}K(x/h)$.
In the univariate setting, the kernel density estimator can then be written as
\[
\widehat f_h(x) = \frac{1}{n}\sum_{i=1}^n K_h(x-X_i).
\]
Equivalently,
\[
\widehat f_h(x) = \frac{1}{nh} \sum_{i=1}^n K\left(\frac{x-X_i}{h}\right).
\]
Taking expectation gives
\begin{equation}\label{eq:expectation-convolution} 
\mathbb{E}\widehat f_h(x) = (K_h*f)(x) = \int_{\mathbb{R}}K(u)f(x-hu)\,du.
\end{equation}

\subsection{$C^{1,1}$ functions and weak second derivatives}
\label{subsec:c11}
We next specify the regularity class used in the univariate analysis. We also make explicit the weak-derivative convention used throughout the paper. 

A function $g\in L^1_{\mathrm{loc}}(\mathbb R)$ is called the weak second derivative of $f$ if
\[
 \int_{\mathbb R} g(x)\varphi(x)\,dx = \int_{\mathbb R} f(x)\varphi''(x)\,dx,
\]
for every smooth compactly supported test function $\varphi$. In this case we write $g=f''$. Weak derivatives are unique up to equality almost everywhere; see, for example \cite{EvansGariepy2015}.

\begin{definition}[\textbf{$C^{1,1}$ smoothness}]
A continuously differentiable function $f:\R\to\R$ belongs to $C^{1,1}(\R)$ if its derivative $f'$ is Lipschitz continuous on $\R$, that is, if there exists $L\ge0$ such that
\begin{equation}\label{eq:Lipschitz-gradient} 
\abs{f'(x)-f'(y)}\le L\abs{x-y}, \qquad \forall x,y\in\R.
\end{equation}
The smallest such constant is denoted by $\operatorname{Lip}(f')$.
\end{definition}

If $f\in C^{1,1}(\mathbb R)$, then $f'$ is locally absolutely continuous. Hence $f'$ is differentiable almost everywhere, and its almost-everywhere derivative coincides with the weak derivative of $f'$. Equivalently, $f$ has a weak second derivative, still denoted by $f''$. Moreover,
\[
 f'(y)-f'(x)=\int_x^y f''(t)\,dt,\qquad x,y\in\mathbb R.
\]
Since $f'$ is Lipschitz continuous with constant $\operatorname{Lip}(f')$, the weak second derivative satisfies
\[
 \|f''\|_{\infty}\le \operatorname{Lip}(f').
\]
Thus a $C^{1,1}$ function has an essentially bounded weak second derivative. However, this $L^\infty$ bound alone does not imply $f''\in L^2(\mathbb R)$ on the whole real line. Therefore, square-integrability of the weak second derivative is imposed separately when integrated risk and AMISE results are derived.

Whenever
\begin{equation}\label{eq:rel_weak_curv}
R(f'') := \int_{\mathbb{R}}\{f''(x)\}^2\,dx,
\end{equation}
appears below, $f''$ is understood as the weak second derivative of $f$. Since functions are identified up to equality almost everywhere, the value assigned to $f''$ on a Lebesgue null set is immaterial for this functional.

\begin{remark}[\textbf{$C^{1,1}$ versus $C^2$ and weak smoothness}]
\label{rem:holder-sobolev}
The condition $C^{1,1}$ should be distinguished from both classical $C^2$ smoothness and Sobolev-type weak smoothness. Functions in $C^{1,1}(\mathbb{R})$ may fail to be twice differentiable at some points, and hence need not belong to $C^2(\mathbb{R})$. On the other hand, the integrated AMISE expansion below only requires the weak-curvature condition $f\in L^2(\mathbb R)$ and $f''\in L^2(\mathbb R)$; this is weaker than global $C^{1,1}$ regularity and is naturally expressed in Sobolev language. The role of $C^{1,1}$ in this paper is to provide pointwise bias control and a concrete nonsmooth class in which curvature is locally bounded but may be kinked or discontinuous. Thus the purpose is to connect the interpretable class $C^{1,1}\setminus C^2$ with the classical AMISE calculation through weak curvature.
\end{remark}

The following integral Taylor representation replaces the classical pointwise $C^2$ Taylor formula with an identity involving the weak second derivative.

\begin{lemma}[\textbf{Integral Taylor formula}]\label{lem:integral-taylor}
Let $f\in C^{1,1}(\R)$. Then, for every $x,v\in\R$,
\begin{equation}\label{eq:integral-taylor} 
f(x+v)=f(x)+vf'(x)+v^2\int_0^1(1-t)f''(x+tv)\,dt,
\end{equation}
where $f''$ denotes any measurable representative of the weak second derivative of $f$.
\end{lemma}
\begin{proof}
Since $f'$ is locally Lipschitz, it is absolutely continuous on every compact interval \cite[Chapter 3]{EvansGariepy2015}. Let $g$ be a measurable representative of the weak derivative of $f'$. Then, 
\[
f(x+v)-f(x)=v\int_0^1 f'(x+tv)\,dt, \qquad \forall x, v\in \R.
\]
Moreover, for each $t\in[0,1]$,
\[
f'(x+tv)-f'(x)=v\int_0^t g(x+sv)\,ds.
\]
Substituting this identity into the previous display gives
\[
\begin{aligned}
    f(x+v)-f(x)-vf'(x)&=v\int_0^1 \left(f'(x+tv)-f'(x)\right)\,dt  \\
    &=v^2\int_0^1\int_0^t g(x+sv)\,ds\,dt.
\end{aligned}
\]
Changing the order of integration over the triangle $\{(s,t):0\le s\le t\le1\}$ yields
\[
f(x+v)-f(x)-vf'(x)=v^2\int_0^1(1-s)g(x+sv)\,ds.
\]
Writing $g=f''$ gives \eqref{eq:integral-taylor}. The formula is unaffected by changing $g$ on a Lebesgue null set, since the affine map $s\mapsto x+sv$ preserves null sets when $v\ne0$; the case $v=0$ is immediate.
\end{proof}

\section{Bias and MISE under $C^{1,1}$ Smoothness}\label{sec:bias-mise}

\subsection{Pointwise bias}
The next result shows that the usual second-order pointwise bias bound does not require a pointwise $C^2$ Taylor expansion. The $C^{1,1}$ condition is sufficient.

\begin{theorem}[\textbf{Pointwise bias under $C^{1,1}$ smoothness}]\label{thm:pointwise-bias}
Assume that $K$ satisfies Assumption \ref{ass:kernel} and that $f\in C^{1,1}(\R)$ with Lipschitz constant $L=\operatorname{Lip}(f')$. Then, for every $x\in\R$ and $h > 0$,
\begin{equation}\label{eq:pointwise-bias-bound} 
\abs{\E\widehat f_h(x)-f(x)}\le\frac{Lh^2}{2}\int_{\R}u^2\abs{K(u)}\,du.
\end{equation}
In particular,
\[
\sup_{x\in\R}\abs{\Bias\{\widehat f_h(x)\}}=O(h^2),
\]
as $h\to 0$.
\end{theorem}
\begin{proof}
By \eqref{eq:expectation-convolution} and $\int_{\mathbb{R}}K(u)\,du=1$,

\[
\E\widehat f_h(x)-f(x)=\int_{\R}K(u)\{f(x-hu)-f(x)\}\,du.
\]
Applying Lemma \ref{lem:integral-taylor} with $v=-hu$ gives
\[
f(x-hu)=f(x)-hu f'(x)+h^2u^2\int_0^1(1-t)f''(x-thu)\,dt.
\]
Using $\int uK(u)\,du=0$, we obtain
\begin{equation}\label{eq:bias-integral-representation} 
\E\widehat f_h(x)-f(x)=h^2\int_{\R}u^2K(u)\int_0^1(1-t)f''(x-thu)\,dt\,du.
\end{equation}
Since $f'$ is Lipschitz with constant $L$, the weak second derivative satisfies $|f''|\le L$ almost everywhere. It follows that
\[
\bigl|\E\widehat{f}_h(x) - f(x)\bigr|\le h^2 L \int_{\R} u^2 |K(u)| \left( \int_0^1 (1-t)\, dt \right) du = \frac{L h^2}{2} \int_{\R} u^2 |K(u)|\, du.
\]
The bound is independent of $x$, and hence the bias is uniformly of order $O(h^2)$.
\end{proof}

\subsection{Integrated squared bias}
The pointwise bias bound in Theorem~\ref{thm:pointwise-bias} only uses the Lipschitz constant of $f'$. To obtain an integrated squared bias bound with a finite weak-curvature constant, we also require the weak second derivative to be square integrable.

\begin{assumption}[\textbf{Density regularity}]\label{ass:density}
The density $f$ satisfies
\begin{enumerate}[label=(F\arabic*)]
    \item\label{ass:density:F1} $f\in C^{1,1}(\R)$;
    \item\label{ass:density:F2} the weak second derivative $f''$ belongs to $L^2(\R)$.
\end{enumerate}
\end{assumption}

\begin{remark}\label{rem:density-regularity}
Condition \ref{ass:density:F2} is essential for integrated risk analysis. While Assumption \ref{ass:density:F1} alone is sufficient for the pointwise $O(h^2)$ bias bound in Theorem~\ref{thm:pointwise-bias}, obtaining a finite constant in the integrated squared bias that involves the curvature functional $R(f'') = \|f''\|_2^2$ requires square-integrability of the weak second derivative.
\end{remark}

\begin{theorem}[\textbf{Integrated squared bias bound}]\label{thm:integrated-bias-bound}
Assume that $K$ satisfies Assumption \ref{ass:kernel} and that $f$ satisfies Assumption \ref{ass:density}. Then
\begin{equation}\label{eq:integrated-bias-bound} 
\int_{\R}\left(\E\widehat f_h(x)-f(x)\right)^2\,dx\le\frac{h^4}{4}\left(\int_{\R}u^2\abs{K(u)}\,du\right)^2\norm{f''}_2^2.
\end{equation}
\end{theorem}
\begin{proof}

For fixed $u$, define $A_{h,u}(x):=\int_0^1(1-t)f''(x-thu)\,dt$. By Minkowski's integral inequality and translation invariance of the $L^2$ norm,
\[
\begin{aligned}
    \|A_{h,u}\|_2&\le\int_0^1(1-t)\|f''(\cdot-thu)\|_2\,dt =\frac12\|f''\|_2.
\end{aligned}
\]
Therefore, again by Minkowski's integral inequality and the integral representation \eqref{eq:bias-integral-representation},
\[
\begin{aligned}
    \|\mathbb E\widehat f_h(x)-f(x) \|_2\le h^2\int_{\mathbb R}u^2|K(u)|\,\|A_{h,u}\|_2\,du \le \frac{h^2}{2}\left(\int_{\mathbb R}u^2|K(u)|\,du\right)\|f''\|_2.
\end{aligned}
\]
Squaring both sides gives \eqref{eq:integrated-bias-bound}.
\end{proof}

\subsection{MISE upper bound}
We first recall a standard bound on the integrated variance term, which requires no smoothness on the density $f$.

\begin{lemma}[\textbf{Integrated variance}]\label{lem:integrated-variance}
Assume that $K\in L^2(\R)$ and that $f$ is a probability density. Then
\begin{equation}\label{eq:integrated-variance-bound} 
\int_{\R} \Var\bigl(\widehat{f}_h(x)\bigr)\, dx \le \frac{R(K)}{nh},
\end{equation}
where $R(K) = \int_{\R} K(u)^2\, du$. More precisely,
\begin{equation}\label{eq:integrated-variance-identity} 
\int_{\mathbb{R}}\Var\bigl(\widehat f_h(x)\bigr)\,dx=\frac{R(K)}{nh}-\frac{1}{n}\|K_h*f\|_2^2.
\end{equation}
If, in addition, $K\in L^1(\R)$ and $f\in L^2(\R)$, then
\begin{equation}\label{eq:integrated-variance-expansion} 
\int_{\R} \Var\bigl(\widehat{f}_h(x)\bigr)\, dx = \frac{R(K)}{nh} + O(n^{-1}),
\end{equation}
where the $O(n^{-1})$ term is bounded uniformly in $h$.
\end{lemma}
\begin{proof}
For fixed $x$, independence of $X_1,\ldots,X_n$ gives
\[
\Var\bigl(\widehat{f}_h(x)\bigr) = \frac{1}{n} \Var\bigl(K_h(x - X_1)\bigr) \le \frac{1}{n} \E\bigl[ K_h(x - X_1)^2 \bigr].
\]
Integrating over $x$ and applying Fubini's theorem yields
\[
\int_{\R} \Var\bigl(\widehat{f}_h(x)\bigr)\, dx \le \frac{1}{n} \int_{\R} \int_{\R} K_h(x-y)^2 f(y)\, dy\, dx = \frac{1}{n} \int_{\R} f(y)\, dy \int_{\R} K_h(z)^2\, dz = \frac{R(K)}{nh}.
\]
For the exact identity, note that $\E\{K_h(x-X_1)\}=(K_h*f)(x)$. Therefore,
\[
\int_{\R} \Var\bigl(\widehat{f}_h(x)\bigr)\, dx = \frac1n\int_{\R}\left[\E\{K_h(x-X_1)^2\} - \{\E K_h(x-X_1)\}^2 \right]dx = \frac{R(K)}{nh} - \frac{1}{n} \|K_h * f\|_2^2.
\]
Under the additional assumptions $K\in L^1(\R)$ and $f\in L^2(\R)$, Young's convolution inequality implies
\[
\|K_h * f\|_2 \le \|K_h\|_1\|f\|_2=\|K\|_1 \|f\|_2.
\]
Hence $\frac1n\|K_h*f\|_2^2=O(n^{-1})$
uniformly in $h$, and \eqref{eq:integrated-variance-expansion} follows.
\end{proof}

\begin{theorem}[\textbf{MISE upper bound under $C^{1,1}$}]\label{thm:mise-bound}
Assume that $K$ satisfies Assumption \ref{ass:kernel} and that $f$ satisfies Assumption \ref{ass:density}. Then, for every $h>0$,
\begin{equation}\label{eq:mise-upper-bound} 
\MISE(\widehat{f}_h) \le \frac{h^4}{4} \Bigl( \int_{\R} u^2 |K(u)|\, du \Bigr)^2 \|f''\|_2^2 + \frac{R(K)}{nh}.
\end{equation}
In particular, $\MISE(\widehat{f}_h) \to 0$ as $h\to 0$ and $nh\to\infty$.
\end{theorem}
\begin{proof}
By the usual bias--variance decomposition,
\[
\MISE(\widehat{f}_h) = \int_{\R} \bigl(\E\widehat{f}_h(x) - f(x)\bigr)^2\, dx + \int_{\R} \Var\bigl(\widehat{f}_h(x)\bigr)\, dx.
\]
The first term is bounded by Theorem \ref{thm:integrated-bias-bound}, and the second term is bounded by Lemma \ref{lem:integrated-variance}. Combining the two bounds gives \eqref{eq:mise-upper-bound}. If $h\to0$ and $nh\to\infty$, then the two terms on the right-hand side of \eqref{eq:mise-upper-bound} both converge to zero.
\end{proof}

\begin{corollary}[\textbf{MISE upper rate}]\label{cor:mise-rate}
Under the assumptions of Theorem \ref{thm:mise-bound}, any bandwidth sequence satisfying $h_n\asymp n^{-1/5}$ yields
\[
\MISE(\widehat{f}_{h_n}) = O(n^{-4/5}).
\]
\end{corollary}
\begin{proof}
Substituting $h_n\asymp n^{-1/5}$ into \eqref{eq:mise-upper-bound} gives $h_n^4=O(n^{-4/5})$ and $(nh_n)^{-1}=O(n^{-4/5})$.
Thus both terms in the upper bound are of order $n^{-4/5}$.
\end{proof}

\section{AMISE Expansion with Weak Second Derivatives}\label{sec:amise}

The upper bounds in Section~\ref{sec:bias-mise} are sufficient to prove consistency, but they do not identify the sharp leading constant in the integrated squared bias. In particular, the bound in Theorem~\ref{thm:integrated-bias-bound} involves $\int u^2|K(u)|\,du$, because it is obtained by taking absolute values. We now show that the usual AMISE expansion remains valid under the weaker integrated condition that the second derivative exists in the weak $L^2$ sense. Thus the sharp AMISE constant is an $L^2$ weak-curvature object; the $C^{1,1}$ assumption is mainly needed for the pointwise bias bounds and for the local nonsmooth-curvature interpretation developed later.

\begin{assumption}[\textbf{Weak AMISE regularity}]\label{ass:weak-amise}
The density $f$ satisfies
\begin{enumerate}[label=(A\arabic*)]
    \item\label{ass:weak-amise:1} $f\in L^2(\mathbb R)$;
    \item\label{ass:weak-amise:2} $f$ has a weak second derivative $f''\in L^2(\mathbb R)$.
\end{enumerate}
Equivalently, $f\in H^2(\mathbb R)$ in the usual Sobolev sense. In particular, $f'$ exists weakly and belongs to $L^2(\mathbb R)$.
\end{assumption}

\begin{lemma}[\textbf{Weak integral Taylor identity in $L^2$}]\label{lem:weak-integral-taylor-L2}
Assume that $f$ satisfies Assumption~\ref{ass:weak-amise}. Then, for every fixed $u\in\mathbb R$ and $h>0$,
\begin{equation}\label{eq:weak-taylor-L2}
 f(\cdot-hu)-f(\cdot)+hu f'(\cdot)
 = h^2u^2\int_0^1(1-t)f''(\cdot-thu)\,dt
\end{equation}
in $L^2(\mathbb R)$.
\end{lemma}
\begin{proof}
Let $\rho_\varepsilon$ be a standard mollifier and set $f_\varepsilon=f*\rho_\varepsilon$. Since $f\in H^2(\mathbb R)$, we have
\[
 f_\varepsilon\to f,\qquad f_\varepsilon'\to f',\qquad f_\varepsilon''\to f''
 \quad\text{in }L^2(\mathbb R).
\]
For the smooth function $f_\varepsilon$, the ordinary integral Taylor formula gives
\[
 f_\varepsilon(\cdot-hu)-f_\varepsilon(\cdot)+hu f_\varepsilon'(\cdot)
 =h^2u^2\int_0^1(1-t)f_\varepsilon''(\cdot-thu)\,dt .
\]
The left-hand side converges in $L^2(\mathbb R)$ by the preceding convergence and translation invariance of the $L^2$ norm. The right-hand side also converges in $L^2(\mathbb R)$, because Minkowski's inequality and translation invariance give
\[
\left\|\int_0^1(1-t)\{f_\varepsilon''(\cdot-thu)-f''(\cdot-thu)\}\,dt\right\|_2
\le \frac12\|f_\varepsilon''-f''\|_2\to0.
\]
Passing to the limit proves~\eqref{eq:weak-taylor-L2}.
\end{proof}

\begin{theorem}[\textbf{AMISE expansion under weak $L^2$ curvature}]\label{thm:amise-expansion}
Assume that $K$ satisfies Assumption \ref{ass:kernel} and that $f$ satisfies Assumption~\ref{ass:weak-amise}. If $h=h_n \to 0$ and $nh_n \to \infty$, then
\begin{equation}\label{eq:mise-expansion} 
\MISE(\widehat{f}_h) = \frac{h^4}{4} \mu_2(K)^2 R(f'') + \frac{R(K)}{nh} + o(h^4) + O(n^{-1}),
\end{equation}
where $\mu_2(K) = \int_{\R} u^2 K(u)\, du$ and $R(f'') = \|f''\|_2^2$. In particular,
\begin{equation}\label{eq:amise-classical-weak} 
\MISE(\widehat{f}_h) = \frac{h^4}{4} \mu_2(K)^2 R(f'') + \frac{R(K)}{nh} + o\bigl(h^4 + (nh)^{-1}\bigr).
\end{equation}
\end{theorem}
\begin{proof}
Let $B_h:=\E\widehat f_h-f=K_h*f-f$. By Lemma~\ref{lem:weak-integral-taylor-L2}, for each fixed $u$,
\[
 f(\cdot-hu)-f(\cdot)
 =-hu f'(\cdot)+h^2u^2\int_0^1(1-t)f''(\cdot-thu)\,dt
\]
in $L^2(\mathbb R)$. Integrating this identity with respect to $K(u)\,du$ and using $\int uK(u)\,du=0$ gives
\[
A_h := \frac{B_h}{h^2} = \int_{\R} u^2 K(u) \left( \int_0^1 (1-t) f''(\cdot - t h u)\, dt \right) du
\]
in $L^2(\mathbb R)$. We prove that
\begin{equation}\label{eq:Ah-L2-convergence} 
A_h \to \frac{\mu_2(K)}{2} f'' \qquad \text{in } L^2(\R) \quad \text{as } h \to 0.
\end{equation}
Indeed,
\[
A_h - \frac{\mu_2(K)}{2} f'' = \int_{\R} u^2 K(u) \left[ \int_0^1 (1-t) \{f''(\cdot - t h u) - f''\}\,dt \right] du.
\]
Taking the $L^2$ norm and applying Minkowski's integral inequality gives
\[
\Bigl\| A_h - \frac{\mu_2(K)}{2} f'' \Bigr\|_2 \le \int_{\R} u^2 |K(u)| \int_0^1 (1-t) \bigl\| f''(\cdot - t h u) - f'' \bigr\|_2 \, dt \, du.
\]
For each fixed $u\in\mathbb{R}$ and $t\in[0,1]$, the $L^2$-continuity of translations gives
\[
\bigl\| f''(\cdot - t h u) - f'' \bigr\|_2 \to 0 \quad \text{as } h \to 0.
\]
Moreover,
\[
u^2|K(u)|(1-t)\|f''(\cdot-thu)-f''\|_2 \le 2u^2|K(u)|(1-t)\|f''\|_2,
\]
and the right-hand side is integrable over $\R\times[0,1]$ by Assumption~\ref{ass:kernel}. The dominated convergence theorem therefore proves \eqref{eq:Ah-L2-convergence}.
It follows that
\[
\|B_h\|_2^2 = h^4 \|A_h\|_2^2 = h^4 \Bigl( \frac{\mu_2(K)}{2} \Bigr)^2 \|f''\|_2^2 + o(h^4) = \frac{h^4}{4} \mu_2(K)^2 R(f'') + o(h^4).
\]
Combining this integrated squared bias expansion with the integrated variance identity from Lemma~\ref{lem:integrated-variance} gives
\[
\MISE(\widehat f_h)=\frac{h^4}{4}\mu_2(K)^2R(f'')+\frac{R(K)}{nh}+o(h^4)+O(n^{-1}).
\]
Finally, since $h\to0$, we have $n^{-1}=h(nh)^{-1}=o((nh)^{-1})$.
Therefore the $O(n^{-1})$ term is negligible relative to $(nh)^{-1}$, and \eqref{eq:amise-classical-weak} follows.
\end{proof}

The expansion in Theorem~\ref{thm:amise-expansion} motivates the usual asymptotic mean integrated squared error approximation
\begin{equation}\label{eq:amise-definition} 
\AMISE(\widehat f_h) :=\frac{h^4}{4}\mu_2(K)^2R(f'')+\frac{R(K)}{nh},
\end{equation}
where $R(f'')=\|f''\|_2^2$ and $f''$ is interpreted as the weak second derivative.

\begin{corollary}[\textbf{AMISE-optimal bandwidth}]\label{cor:amise-bandwidth}
Assume that $\mu_2(K)^2 R(f'') > 0$. Then the bandwidth that minimizes the AMISE \eqref{eq:amise-definition} is given by
\begin{equation}\label{eq:h-amis} 
h_{\AMISE} = \left( \frac{R(K)}{\mu_2(K)^2 R(f'') \, n} \right)^{1/5}.
\end{equation}
In particular, $ h_{\AMISE} \asymp n^{-1/5} $ whenever $ R(K) $ and $ R(f'') $ are fixed positive constants.
\end{corollary}
\begin{proof}
Differentiating \eqref{eq:amise-definition} with respect to $h$ gives
\[
\frac{d}{dh} \AMISE(\widehat{f}_h) = h^3 \mu_2(K)^2 R(f'') - \frac{R(K)}{n h^2}.
\]
Setting the derivative equal to zero yields
\[
h^3 \mu_2(K)^2 R(f'') = \frac{R(K)}{n h^2},
\]
which simplifies to
\[
h^5 = \frac{R(K)}{\mu_2(K)^2 R(f'') \, n}.
\]
Solving for $ h $ gives the claimed expression \eqref{eq:h-amis}.
\end{proof}

\begin{remark}[\textbf{Degenerate cases}]\label{rem:degenerate-curvature}
The non-degeneracy condition $\mu_2(K)^2 R(f'') > 0$ excludes two important cases in which the standard second-order AMISE formula is inappropriate:

\begin{itemize}
    \item If $\mu_2(K) = 0$, then $K$ is a higher-order kernel and the leading $h^4$ bias term disappears. In this case, a different expansion and bandwidth formula are needed.

    \item If $R(f'') = 0$, then $f'' = 0$ almost everywhere. For an integrable density on $\R$, this implies that $f$ is affine almost everywhere, which is incompatible with $f$ being a non-degenerate probability density (since $\int_{\R} f = 1$ and $f \geq 0$).
\end{itemize}

In both cases, the classical second-order bandwidth selector should not be used without an appropriate modification of the asymptotic expansion.
\end{remark}

\section{Generalized-Curvature Plug-In Bandwidth Selection}\label{sec:gcpi}

The AMISE bandwidth in Corollary \ref{cor:amise-bandwidth} depends on the unknown weak-curvature functional $R(f'')$. Classical plug-in bandwidth selectors replace this unknown quantity by a pilot estimate; see \cite{Chiu1996,HallSheatherJonesMarron1991,Heidenreich2013,JonesMarronSheather1996,ParkMarron1990,Sheather2004,SheatherJones1991,WandJones1995}. Density-derivative estimation and derivative-based bandwidth selectors are also treated in \cite{ChaconDuong2013,ChaconDuongWand2011,FanMarron1992}, and in modern computational work such as \cite{Guidoum2020}. Integrated squared derivative estimation, including estimation of functionals of the form $R(f^{(r)})$, was studied in \cite{BickelRitov1988,HallMarron1987,JonesSheather1991}.

In the present setting, the same plug-in principle is used, but the target curvature functional is interpreted weakly rather than through a pointwise continuous classical second derivative. This section formulates a proof-of-concept weak-curvature plug-in rule. The aim is to show that the classical first-order AMISE plug-in logic remains meaningful when the curvature term is $R(f'')$ with $f''$ understood as a weak second derivative.

Let $L$ be a pilot kernel satisfying $L\in L^1(\mathbb R)$, $\int_{\mathbb R}L(u)\,du=1$, and suppose that its weak second derivative satisfies $L''\in L^1(\mathbb R)\cap L^2(\mathbb R)$. Let $b>0$ be a pilot bandwidth and define $L_b(x):=b^{-1}L(x/b)$. Then $L_b''(x)=b^{-3}L''(x/b)$ in the weak sense. The pilot estimator of the second derivative of $f$ is
\begin{equation}\label{eq:pilot-second-derivative} 
\widehat{f}_b''(x) := \frac{1}{n} \sum_{i=1}^n L_b''(x - X_i) = \frac{1}{n b^3} \sum_{i=1}^n L''\Bigl( \frac{x - X_i}{b} \Bigr).
\end{equation}
Under the standing assumption $f''\in L^2(\mathbb{R})$, its expectation satisfies
\[
\E[\widehat{f}_b''(x)] = L_b'' * f = L_b * f'' \qquad \text{in } L^2(\R),
\]
where the last equality holds in the sense of distributions (and hence in $ L^2 $) under our assumptions.

A natural estimator of $ R(f'') $ is $ \int_{\R} \{\widehat{f}_b''(x)\}^2 \, dx $. However, this quantity contains diagonal self-interaction terms. To isolate the off-diagonal part, define the autocorrelation kernel
\begin{equation}\label{eq:pilot-autocorrelation} 
G_L(z):=\int_{\mathbb{R}} L''(y)L''(y+z)\,dy.
\end{equation}
Equivalently, if $\widetilde{L''}(y)=L''(-y)$, then $G_L=\widetilde{L''}*L''$. If $L$ is even, then $L''$ is even and this reduces to $G_L=L''*L''$. Since $L''\in L^1(\mathbb R)\cap L^2(\mathbb R)$, Young's convolution inequality gives $G_L\in L^1(\mathbb R)\cap L^2(\mathbb R)$.

We define the off-diagonal U-statistic-type curvature estimator by
\begin{equation}\label{eq:U-curvature-estimator}
\widehat{R}_b^U := \frac{1}{n(n-1) b^5} \sum_{1 \le i \ne j \le n} G_L\Bigl( \frac{X_i - X_j}{b} \Bigr).
\end{equation}

The following lemma makes the relationship between the two estimators explicit.
\begin{lemma}[\textbf{Diagonal decomposition}]\label{lem:diagonal-decomposition}
Assume $ L'' \in L^2(\R) $. Then
\begin{equation}
\label{eq:diagonal-decomposition} 
\int_{\R} \{\widehat{f}_b''(x)\}^2 \, dx = \frac{n-1}{n} \widehat{R}_b^U + \frac{1}{n b^5} R(L''),
\end{equation}
where $ R(L'') = \int_{\R} [L''(u)]^2 \, du $.
\end{lemma}
\begin{proof}
By definition,
\[
\int_{\R} \{\widehat{f}_b''(x)\}^2 \, dx = \frac{1}{n^2} \sum_{i=1}^n \sum_{j=1}^n \int_{\R} L_b''(x - X_i) L_b''(x - X_j) \, dx.
\]
For $i\ne j$, using $L_b''(x)=b^{-3}L''(x/b)$ and the change of variables $ y = (x - X_i)/b $ gives
\[
\int L_b''(x - X_i) L_b''(x - X_j) \, dx = b^{-5} \int_{\R} L''(y) L''\Bigl(y + \frac{X_i - X_j}{b}\Bigr) \, dy = b^{-5} G_L\Bigl( \frac{X_i - X_j}{b} \Bigr).
\]
For the diagonal terms ($ i = j $), the same calculation gives
\[
\int_{\R} [L_b''(x - X_i)]^2 \, dx = b^{-5} R(L'').
\]
Therefore,
\[
\begin{aligned}
    \int_{\mathbb{R}}\{\widehat f_b''(x)\}^2\,dx  &= \frac{1}{n^2b^5} \sum_{1\le i\ne j\le n} G_L\left(\frac{X_i-X_j}{b}\right)
    + \frac{1}{n^2} \sum_{i=1}^n b^{-5}R(L'')       \\
    &= \frac{n-1}{n}\widehat R_b^U + \frac{1}{nb^5}R(L'').
\end{aligned}
\]
This proves~\eqref{eq:diagonal-decomposition}.
\end{proof}

Because finite-sample values of \eqref{eq:U-curvature-estimator} may occasionally be nonpositive, especially for small samples or hard-to-estimate densities, we use the truncated version
\begin{equation}\label{eq:truncated-curvature-v5} 
\widehat R_{b,\tau}:=\max\{\widehat R_b^{\,U},\tau_n\},
\end{equation}
where $\tau_n\downarrow0$ is a deterministic numerical safeguard. The generalized-curvature plug-in bandwidth is then defined by
\begin{equation}\label{eq:gcpi-bandwidth-v5} 
\widehat h_{\operatorname{GCPI}} :=\left\{\frac{R(K)}{\mu_2(K)^2\,\widehat R_{b,\tau}\, n}\right\}^{1/5}.
\end{equation}
The factor $n$ in the denominator of \eqref{eq:gcpi-bandwidth-v5} is essential: it gives the plug-in bandwidth the same $n^{-1/5}$ scale as the AMISE oracle bandwidth. We call \eqref{eq:gcpi-bandwidth-v5} the generalized-curvature plug-in (GCPI) bandwidth. The term ``generalized'' refers to the fact that the target curvature is $R(f'')$ with $f''$ interpreted as a weak second derivative.

\medskip
\noindent\fbox{%
\begin{minipage}{0.96\linewidth}
\textbf{Algorithm 1: Generalized-curvature plug-in bandwidth selector.}

\smallskip
\noindent\textbf{Input:} data $X_1,\ldots,X_n$; main kernel $K$; pilot kernel $L$; pilot bandwidth $b=b_n$; truncation level $\tau_n>0$.

\begin{enumerate}[leftmargin=1.5em]
    \item Compute the pilot second-derivative estimator $\widehat f_b''$ in \eqref{eq:pilot-second-derivative}.
    \item Compute the off-diagonal curvature estimate $\widehat R_b^{\,U}$ in \eqref{eq:U-curvature-estimator}, equivalently by the diagonal correction in Lemma~\ref{lem:diagonal-decomposition} when numerical quadrature is used.
    \item Truncate the curvature estimate by $\widehat R_{b,\tau}=\max\{\widehat R_b^{\,U},\tau_n\}$.
    \item Return
    \[
    \widehat h_{\operatorname{GCPI}}
    =\left\{\frac{R(K)}{\mu_2(K)^2\widehat R_{b,\tau}n}\right\}^{1/5}.
    \]
\end{enumerate}
\end{minipage}}

\medskip
For a practical default in the numerical examples, we use a Gaussian pilot kernel and a scale-adjusted pilot bandwidth of the form
\[
 b_n=s_{\rm rob} n^{-1/9},\qquad s_{\rm rob}=\min\{s,\operatorname{IQR}/1.34\},
\]
where $s$ is the sample standard deviation and $\operatorname{IQR}$ is the sample interquartile range. This rule is not claimed to be optimal; it is a stable implementation choice that keeps the pilot bandwidth explicit and reproducible. The asymptotic theory below only requires bandwidth conditions sufficient for ratio consistency of the weak-curvature estimator.

The next theorem isolates the essential statistical requirement. It is stated in a modular form: any pilot curvature estimator satisfying ratio consistency can be inserted into the AMISE bandwidth formula. The specific U-statistic estimator \eqref{eq:U-curvature-estimator} is treated in Proposition~\ref{prop:U-consistency-detailed}.
\begin{theorem}[\textbf{Oracle equivalence of the GCPI bandwidth}]\label{thm:gcpi-oracle}
Assume the conditions of Theorem \ref{thm:amise-expansion}, and suppose $\mu_2(K)^2R(f'')>0$. Let $\widehat R_n$ be a nonnegative estimator of $R(f'')$ such that
\begin{equation}
\label{eq:ratio-consistency-curvature-v5} 
\frac{\widehat R_n}{R(f'')}\xrightarrow{p}1.
\end{equation}
On the event $\{\widehat R_n>0\}$, define
\begin{equation}
\label{eq:generic-plugin-bandwidth-v5} 
\widehat h_n=\left\{\frac{R(K)}{\mu_2(K)^2\,\widehat R_n\, n}\right\}^{1/5}.
\end{equation}
On the event $\{\widehat R_n=0\}$, define $\widehat h_n=h_{\AMISE}$. Then
\begin{equation}
\label{eq:plugin-ratio-v5} 
\frac{\widehat h_n}{h_{\AMISE}}\xrightarrow{p}1.
\end{equation}
Moreover,
\begin{equation}
\label{eq:plugin-risk-ratio-v5} 
\frac{\AMISE(\widehat f_{\widehat h_n})}{\AMISE(\widehat f_{h_{\AMISE}})}\xrightarrow{p}1.
\end{equation}
The same conclusion holds for the truncated bandwidth \eqref{eq:gcpi-bandwidth-v5} whenever $\widehat R_b^{\,U}/R(f'')\to_p1$ and $\tau_n/R(f'')\to0$.
\end{theorem}
\begin{proof}
By \eqref{eq:h-amis} and \eqref{eq:generic-plugin-bandwidth-v5}, on the event $\{\widehat R_n>0\}$,
\[
\frac{\widehat h_n}{h_{\AMISE}} =\left\{\frac{R(f'')}{\widehat R_n}\right\}^{1/5}.
\]
Since $\widehat R_n/R(f'')\xrightarrow{p}1$ and the map $x\mapsto x^{-1/5}$ is continuous in a neighborhood of $1$, the continuous mapping theorem gives
\[
\frac{\widehat h_n}{h_{\AMISE}}\xrightarrow{p} 1.
\]
The definition on $\{\widehat R_n=0\}$ does not affect the conclusion, because $\Pr(\widehat R_n=0)\to0$.

Set
\[
A:=\frac{R(K)}{n h_{\AMISE}}, \qquad B:=\frac{h_{\AMISE}^4}{4}\mu_2(K)^2R(f'').
\]
The first-order condition defining $h_{\AMISE}$ is
\[
h_{\AMISE}^5\mu_2(K)^2R(f'')=\frac{R(K)}{n}.
\]
Dividing both sides by $h_{\AMISE}$ gives
\[
h_{\AMISE}^4\mu_2(K)^2R(f'') = \frac{R(K)}{nh_{\AMISE}},
\]
and hence $A=4B$. Let $a=\widehat h_n/h_{\AMISE}$, then
\[
\AMISE(\widehat f_{a h_{\AMISE}})=Aa^{-1}+Ba^4, \qquad \AMISE(\widehat f_{h_{\AMISE}})=A+B=5B.
\]
Consequently,
\begin{equation}
\label{eq:amise-ratio-a-v5} 
\frac{\AMISE(\widehat f_{a h_{\AMISE}})} {\AMISE(\widehat f_{h_{\AMISE}})} =\frac{4a^{-1}+a^4}{5}.
\end{equation}
Since $a\to_p1$ by \eqref{eq:plugin-ratio-v5}, another application of the continuous mapping theorem gives \eqref{eq:plugin-risk-ratio-v5}.

For the truncated version, if $\widehat R_b^{\,U}\to_p R(f'')$ and $R(f'')>0$, then
\[
\Pr\left\{\widehat R_b^{\,U}<\tau_n\right\} \le \Pr\left\{\abs{\widehat R_b^{\,U}-R(f'')} > R(f'')-\tau_n\right\}\to0,
\]
for all sufficiently large $n$. Hence $\widehat R_{b,\tau}/R(f'')\to_p1$, and the preceding argument applies to the truncated GCPI bandwidth.
\end{proof}

\begin{proposition}[\textbf{Consistency of the U-statistic curvature estimator}]\label{prop:U-consistency-detailed}
Assume the following conditions.
\begin{enumerate}[label=(P\arabic*)]
    \item\label{ass:pilot:P1}
$L\in L^1(\mathbb{R})$, $\int_{\mathbb{R}}L(u)\,du=1$, and the weak second derivative $L''$ belongs to $L^1(\mathbb{R})\cap L^2(\mathbb{R})$.

    \item\label{ass:pilot:P2}
The autocorrelation kernel $G_L$ defined in~\eqref{eq:pilot-autocorrelation} belongs to $L^1(\mathbb{R})\cap L^2(\mathbb{R})$.

    \item\label{ass:pilot:P3}
The density $f$ is bounded and has weak second derivative $f''\in L^2(\mathbb{R})$, with $R(f'')>0$.

    \item\label{ass:pilot:P4} The pilot bandwidth satisfies
    \begin{equation}
\label{eq:pilot-bandwidth-conditions-v5} 
b=b_n\to0, \qquad n b_n^4\to\infty, \qquad n^2 b_n^9\to\infty.
\end{equation}
\end{enumerate}
Then
\begin{equation}
\label{eq:U-consistency-v5}
\widehat R_b^{\,U}\xrightarrow{p}R(f'').
\end{equation}
Consequently, $\frac{\widehat R_b^{\,U}}{R(f'')}\xrightarrow{p} 1$. Under the assumptions of Theorem~\ref{thm:amise-expansion}, with $\mu_2(K)\ne 0$, the truncated GCPI bandwidth~\eqref{eq:gcpi-bandwidth-v5} is oracle-equivalent in the sense of Theorem~\ref{thm:gcpi-oracle} for any $\tau_n\downarrow0$ satisfying $\tau_n/R(f'')\to0$. 

\end{proposition}
\begin{proof}
Define $\psi_b(x):=L_b''(x)=b^{-3}L''(x/b)$ and
\[
H_b(x,y):=\int_{\R}\psi_b(t-x)\psi_b(t-y)\,dt =b^{-5}G_L\left(\frac{x-y}{b}\right).
\]
Then
\[
\widehat R_b^{\,U}=\frac{1}{n(n-1)}\sum_{i\ne j}H_b(X_i,X_j),
\]
is a symmetric $U$-statistic of order two with kernel $H_b$ depending on $n$ through $b=b_n$.
\\
\emph{Step 1: Expectation.} Let
\[
m_b(t):=\E\{\psi_b(t-X_1)\}=\int_{\mathbb{R}}\psi_b(t-x)f(x)\,dx.
\]
Because convolution commutes with weak differentiation,
\[
m_b=L_b''*f=L_b*f'',
\]
in $L^2(\R)$. Using independence of $X_1$ and $X_2$, we obtain
\begin{align*}
    \E\widehat R_b^{\,U} =\E H_b(X_1,X_2) &=\E\int \psi_b(t-X_1)\psi_b(t-X_2)\,dt \\
    &=\int \E\{\psi_b(t-X_1)\}\E\{\psi_b(t-X_2)\}\,dt 
    =\int m_b(t)^2\,dt
    =\norm{L_b*f''}_2^2.
\end{align*}
Since $L_b$ is an approximate identity and $f''\in L^2(\R)$, we have $L_b*f''\to f'' $ in $L^2(\R)$.
It follows that
\begin{equation}
\label{eq:expectation-converges-v5} 
\E\widehat R_b^{\,U}=\norm{L_b*f''}_2^2\to \norm{f''}_2^2=R(f'').
\end{equation}
Indeed,
\[
\left|\norm{L_b*f''}_2^2-\norm{f''}_2^2\right| \le \norm{L_b*f''-f''}_2\{\norm{L_b*f''}_2+\norm{f''}_2\}\to0.
\]
\emph{Step 2: Variance.} For a symmetric order-two $U$-statistic with kernel depending on $n$, the standard variance decomposition gives
\begin{equation}
\label{eq:Uvariance-bound-v5} 
\Var(\widehat R_b^{\,U}) \le \frac{4}{n}\E\{h_b(X_1)^2\} +\frac{2}{n(n-1)}\E\{H_b(X_1,X_2)^2\},
\end{equation}
where $h_b(x):=\E\{H_b(x,X_2)\}$. We now bound the two terms on the right-hand side.

First,
\[
h_b(x)=\int_{\mathbb{R}}\psi_b(t-x)m_b(t)\,dt.
\]
Hence, by Young's convolution inequality, $\norm{h_b}_2\le \norm{\psi_b}_1\norm{m_b}_2$.
Since
\[
\norm{\psi_b}_1=b^{-2}\norm{L''}_1, \qquad \norm{m_b}_2=\norm{L_b*f''}_2\le \norm{L}_1\norm{f''}_2,
\]
we get
\[
\norm{h_b}_2^2\le b^{-4}\norm{L''}_1^2\norm{L}_1^2\norm{f''}_2^2.
\]
Because $f$ is bounded,
\begin{equation}
\label{eq:projection-bound-v5} 
\E\{h_b(X_1)^2\} =\int_{\mathbb{R}}h_b(x)^2f(x)\,dx \le \norm{f}_{\infty}\norm{h_b}_2^2 \le C b^{-4},
\end{equation}
for a constant $C$ independent of $n$.

Second, using the definition of $H_b$ and boundedness of $f$,
\begin{align}
    \E\{H_b(X_1,X_2)^2\}
    &=b^{-10}\int_{\mathbb{R}}\int_{\mathbb{R}} G_L\left(\frac{x-y}{b}\right)^2f(x)f(y)\,dx\,dy \\
    &\le b^{-10}\norm{f}_{\infty}\int_{\mathbb{R}} f(x)\left\{\int_{\mathbb{R}} G_L\left(\frac{x-y}{b}\right)^2dy\right\}dx \\
    &= b^{-9}\norm{f}_{\infty}\norm{G_L}_2^2.
\end{align}
Combining this bound with \eqref{eq:projection-bound-v5} and \eqref{eq:Uvariance-bound-v5}, we get
\[
\Var(\widehat R_b^{\,U}) \le C_1\frac{1}{n b^4}+C_2\frac{1}{n^2 b^9}.
\]
The pilot-bandwidth conditions \eqref{eq:pilot-bandwidth-conditions-v5} imply
\begin{equation}
\label{eq:variance-converges-v5} 
\Var(\widehat R_b^{\,U})\to0.
\end{equation}
\emph{Step 3: Consistency.} By Chebyshev's inequality,
\[
\widehat R_b^{\,U}-\E\widehat R_b^{\,U}\xrightarrow{p}0.
\]
Together with \eqref{eq:expectation-converges-v5}, this proves \eqref{eq:U-consistency-v5}. Since $R(f'')>0$, ordinary consistency implies ratio consistency:
\[
\frac{\widehat R_b^{\,U}}{R(f'')}\xrightarrow{p}1.
\]
The final statement follows by applying Theorem \ref{thm:gcpi-oracle} to the ratio-consistent curvature estimator and the corrected plug-in formula \eqref{eq:gcpi-bandwidth-v5}.
\end{proof}

\begin{remark}[\textbf{Role and strength of the pilot assumptions}]
\label{rem:pilot-assumptions}
The assumptions in Proposition~\ref{prop:U-consistency-detailed} are sufficient conditions for a transparent consistency proof and are not claimed to be sharp.
\begin{itemize}
    \item Condition~\ref{ass:pilot:P1} ensures that the pilot kernel is normalized and that the derivative estimator is well defined. Since
\[
 G_L=\widetilde{L''}*L'',\qquad \widetilde{L''}(y)=L''(-y),
\]
Young's convolution inequality implies $G_L\in L^1(\mathbb R)\cap L^2(\mathbb R)$ whenever $L''\in L^1(\mathbb R)\cap L^2(\mathbb R)$. This is the regularity needed for the moment and second-moment bounds involving the off-diagonal kernel.

\item  Condition~\ref{ass:pilot:P2} uses boundedness of $f$ to control the variance terms. The condition $f''\in L^2(\mathbb R)$ identifies the target weak-curvature functional $R(f'')$, while $R(f'')>0$ is needed for ratio consistency and for the nondegenerate plug-in bandwidth formula.

\item Condition~\ref{ass:pilot:P3} requires the pilot bandwidth to decrease slowly enough. The condition $b_n\to0$ makes the pilot smoothing bias vanish. The conditions $nb_n^4\to\infty$ and $n^2b_n^9\to\infty$ force the two variance bounds in the proof to vanish. For example, if $b_n=n^{-\alpha}$, then these conditions hold whenever $0<\alpha<\frac{2}{9}$.
Indeed, $nb_n^4=n^{1-4\alpha}\to\infty$ requires $\alpha<1/4$, while $n^2b_n^9=n^{2-9\alpha}\to\infty$ requires $\alpha<2/9$.
\end{itemize}

Sharper assumptions, optimal pilot-bandwidth rates, limiting distributions, and fully data-driven pilot selection are not pursued here. Those refinements belong to the broader integrated-squared-derivative and plug-in bandwidth literature, whereas the present goal is only to show that the classical plug-in AMISE principle can be formulated for the weak-curvature functional $R(f'')$.
\end{remark}

\subsection{Position relative to standard bandwidth selectors}

Table~\ref{tab:selector-comparison} compares the proposed GCPI selector with common bandwidth-selection methods. Its purpose is to clarify the specific role of GCPI: it is not presented as a universally superior practical rule, but as a theoretically justified plug-in selector under the weak-curvature assumptions considered in this paper.

\begin{table}[ht]
\centering
\caption{Comparison of bandwidth selectors. ``Weak curvature'' means that the functional $R(f'')$ is interpreted with $f''$ as a weak second derivative.}
\label{tab:selector-comparison}
\begin{tabularx}{\linewidth}{@{}>{\raggedright\arraybackslash}p{0.18\linewidth}
                            >{\raggedright\arraybackslash}X
                            >{\raggedright\arraybackslash}X
                            >{\raggedright\arraybackslash}X@{}}
\toprule
Selector & Main principle & Relation to curvature & Main limitation \\
\midrule
Normal-reference&Gaussian reference model&Implicit parametric curvature approximation
&Can oversmooth multimodal or non-Gaussian densities
\\
\addlinespace
Least-squares CV&Empirical estimation of integrated squared error
&No explicit curvature estimation&Can be highly variable and may undersmooth
\\
\addlinespace
Classical plug-in&Plug-in estimation of the AMISE curvature term
&Targets quantities such as $R(f'')$&Usually formulated through ordinary derivative or smooth-pilot assumptions
\\
\addlinespace
GCPI&U-statistic estimation of weak curvature&Targets $R(f'')$ interpreted weakly
&Requires pilot bandwidth, second-derivative estimation, and truncation
\\
\bottomrule
\end{tabularx}
\end{table}

This comparison underscores that the main contribution lies not in the plug-in methodology itself, which is classical, but in showing that the standard first-order AMISE plug-in logic can be formulated with the weak-curvature functional $R(f'')$. The $C^{1,1}\setminus C^2$ examples then show how this weak-curvature target arises in concrete nonsmooth densities with Lipschitz first derivatives.

\section{Multivariate Extension Under Weak Hessians}
\label{sec:multivariate}

The preceding analysis was developed first in one dimension to make the role of the weak second derivative transparent. Multivariate kernel smoothing is a substantial subject in its own right; see \cite{ChaconDuong2010,ChaconDuong2013,ChaconDuong2018,ChaconDuongWand2011,Scott2015,WandJones1994} for multivariate bandwidth, smoothing, and density-derivative theory. The multivariate analogue is also the setting in which the present regularity formulation is naturally expressed through weak Hessians. The extension below is deliberately restricted to scalar bandwidths; it is not intended to compete with full bandwidth-matrix plug-in theory.

Let $X_1,\ldots,X_n$ be independent observations in $\mathbb{R}^d$ with density $f:\mathbb{R}^d\to[0,\infty)$. For a scalar bandwidth $h>0$ and kernel $K:\R^d\to\R$, define
\begin{equation}
\label{eq:multivariate-kde} 
\widehat f_h(x):=\frac1{nh^d}\sum_{i=1}^n K\left(\frac{x-X_i}{h}\right), \qquad x\in\R^d.
\end{equation}
We assume
\begin{equation}
\label{eq:multivariate-kernel-moments} 
\int_{\R^d}K(u)\,du=1, \qquad \int_{\R^d}uK(u)\,du=0, \qquad \int_{\R^d}\|u\|^2|K(u)|\,du<\infty,
\end{equation}
with $K\in L^1(\R^d)\cap L^2(\R^d)$. Define the second moment matrix
\begin{equation}
\label{eq:M2K} 
M_2(K):=\int_{\R^d}uu^\top K(u)\,du.
\end{equation}
For an isotropic second-order kernel, $M_2(K)=m_2(K)I_d$.

\begin{definition}[\textbf{$C^{1,1}$ smoothness in $\R^d$}]
A differentiable function $f:\R^d\to\R$ belongs to $C^{1,1}(\R^d)$ if its gradient is Lipschitz continuous, that is, if there exists $L\ge0$ such that
\[
\|\nabla f(x)-\nabla f(y)\|\le L\|x-y\|, \qquad x,y\in\R^d.
\]
Then the weak Hessian $D^2f$ is represented almost everywhere by an essentially bounded matrix-valued function, and this representative satisfies $\|D^2f(x)\|_{\mathrm{op}}\le L$ almost everywhere.
\end{definition}

For pointwise estimates, we use the elementary consequence of the Lipschitz continuity of the gradient:
\begin{equation}
\label{eq:multi-lipschitz-remainder} 
\big|f(x+v)-f(x)-\nabla f(x)^\top v\big| \le \frac{L}{2}\|v\|^2, \qquad x,v\in\R^d.
\end{equation}
For the integrated AMISE expansion, the corresponding second-order identity is used in an $L^2$ sense. The following lemma records the precise form needed below.

\begin{lemma}[\textbf{Directional weak Taylor identity}]
\label{lem:multi-directional-weak-taylor}
Let $f\in C^{1,1}(\mathbb R^d)$ and suppose that $D^2f\in L^2(\mathbb R^d)$, meaning that all weak second partial derivatives belong to $L^2(\mathbb R^d)$. Then, for every fixed $u\in\mathbb R^d$ and $h>0$,
\begin{equation}
\label{eq:multi-weak-directional-taylor} 
f(\cdot-hu)-f(\cdot)+h u^\top \nabla f(\cdot) =h^2\int_0^1(1-t)\,u^\top D^2f(\cdot-thu)u\,dt,
\end{equation}
in $L^2(\mathbb R^d)$.
\end{lemma}
\begin{proof}
Let $\rho_\varepsilon$ be a standard mollifier and set $f_\varepsilon=f*\rho_\varepsilon$. Since $f\in C^{1,1}(\mathbb R^d)$, we have $f_\varepsilon\to f$ and $\nabla f_\varepsilon\to\nabla f$ locally uniformly. Since $D^2f\in L^2(\mathbb R^d)$, the weak Hessians satisfy
\[
D^2f_\varepsilon=(D^2f)*\rho_\varepsilon \to D^2f \qquad\text{in }L^2(\mathbb R^d).
\]
For the smooth function $f_\varepsilon$, the classical integral Taylor formula gives, for every $x\in\mathbb R^d$,
\[
f_\varepsilon(x-hu)-f_\varepsilon(x)+h u^\top\nabla f_\varepsilon(x) = h^2\int_0^1(1-t)\, u^\top D^2f_\varepsilon(x-thu)u\,dt.
\]
We pass to the limit in $L^2(\mathbb R^d)$ on the right-hand side and in the sense of distributions on the left-hand side. For the right-hand side, Minkowski's inequality and translation invariance of the $L^2$ norm give
\[
\begin{aligned}
&\left\|
\int_0^1(1-t)\,
u^\top\{D^2f_\varepsilon(\cdot-thu)-D^2f(\cdot-thu)\}u\,dt
\right\|_2  \\
&\qquad\le
\int_0^1(1-t)
\|u^\top\{D^2f_\varepsilon(\cdot-thu)-D^2f(\cdot-thu)\}u\|_2\,dt \\
&\qquad\le
\frac12\|u\|^2\|D^2f_\varepsilon-D^2f\|_{F,2}
\to0.
\end{aligned}
\]
The left-hand side also converges in the sense of distributions to
\[
f(\cdot-hu)-f(\cdot)+h u^\top\nabla f(\cdot).
\]
The $L^2$ bound
\[
\left\| \int_0^1(1-t)\,u^\top D^2f(\cdot-thu)u\,dt \right\|_2 \le \frac12\|u\|^2\|D^2f\|_{F,2},
\]
shows that the limiting right-hand side belongs to $L^2(\mathbb R^d)$. Therefore the distributional limit on the left is represented by this $L^2$ function, and the displayed identity holds in $L^2(\mathbb R^d)$.
\end{proof}

\begin{theorem}[\textbf{Multivariate pointwise bias bound}]\label{thm:multi-pointwise-bias}
Let $f\in C^{1,1}(\R^d)$ and let $\nabla f$ have Lipschitz constant $L$. Suppose $K$ satisfies \eqref{eq:multivariate-kernel-moments}. Then, for every $x\in\mathbb{R}^d$ and $h>0$,
\begin{equation}
\label{eq:multi-pointwise-bias-bound} 
\left|\mathbb E\widehat f_h(x)-f(x)\right| \le \frac{Lh^2}{2}\int_{\R^d}\|u\|^2|K(u)|\,du.
\end{equation}
\end{theorem}
\begin{proof}
The expectation is
\[
\mathbb E\widehat f_h(x)=\int_{\R^d}K(u)f(x-hu)\,du.
\]
By \eqref{eq:multi-lipschitz-remainder} with $v=-hu$,
\[
f(x-hu)=f(x)-h u^\top\nabla f(x)+r_h(x,u), \qquad |r_h(x,u)|\le \frac{Lh^2}{2}\|u\|^2.
\]
The first-order term vanishes because $\int uK(u)\,du=0$. Hence
\[
|\mathbb E\widehat f_h(x)-f(x)| \le \frac{Lh^2}{2}\int_{\R^d}\|u\|^2|K(u)|\,du,
\]
which gives \eqref{eq:multi-pointwise-bias-bound}.
\end{proof}

For exact integrated constants, assume that all weak second partial derivatives belong to $L^2(\R^d)$. Let
\[
\|D^2f\|_{F,2}^2=\int_{\R^d}\|D^2f(x)\|_F^2\,dx.
\]
The leading bias operator is
\begin{equation}
\label{eq:LK-operator} 
\mathcal L_K f(x)=\tr\{M_2(K)D^2f(x)\}.
\end{equation}
Since $M_2(K)$ is finite and $D^2f\in L^2(\mathbb R^d)$, we have $\mathcal L_K f\in L^2(\mathbb R^d)$. If $M_2(K)=m_2(K)I_d$, then $\mathcal L_K f=m_2(K)\Delta f$.

\begin{theorem}[\textbf{Multivariate AMISE expansion}]
\label{thm:multi-amise}
Assume that $K\in L^1(\mathbb{R}^d)\cap L^2(\mathbb{R}^d)$ satisfies \eqref{eq:multivariate-kernel-moments}, and that $f\in C^{1,1}(\mathbb{R}^d)\cap L^2(\mathbb{R}^d)$ has weak Hessian $D^2f\in L^2(\mathbb{R}^d)$. If $h\to0$ and $nh^d\to\infty$, then
\begin{equation}
\label{eq:multi-mise-expansion} \MISE(\widehat f_h)=\frac{h^4}{4}\|\mathcal L_K f\|_2^2+\frac{R(K)}{nh^d}+ o\{h^4+(nh^d)^{-1}\}.
\end{equation}
\end{theorem}
\begin{proof}
Let $B_h(x)=\mathbb{E}\widehat f_h(x)-f(x)$. Using Lemma~\ref{lem:multi-directional-weak-taylor} under the kernel integral and the moment condition $\int_{\mathbb{R}^d}uK(u)\,du=0$, we obtain
\[
\frac{B_h(x)}{h^2}=\int_{\mathbb{R}^d}K(u)\int_0^1(1-t)u^\top D^2f(x-thu)u\,dt\,du,
\]
in $L^2(\mathbb{R}^d)$. We claim that
\[
\frac{B_h}{h^2}\to \frac12\mathcal L_K f\qquad \text{in } L^2(\mathbb{R}^d).
\]
Indeed,
\[
\begin{aligned}
&\left\|\frac{B_h}{h^2}-\frac12\mathcal L_K f\right\|_2\\
&\quad\le\int_{\mathbb{R}^d}|K(u)|\int_0^1(1-t)\left\|u^\top\{D^2f(\cdot-thu)-D^2f(\cdot)\}u\right\|_2\,dt\,du.
\end{aligned}
\]
For each fixed $u$ and $t$, translation continuity in $L^2(\mathbb{R}^d)$ for each weak second partial derivative gives
\[
\left\|u^\top\{D^2f(\cdot-thu)-D^2f(\cdot)\}u\right\|_2\to 0 \qquad \text{as }h\to0.
\]
Moreover,
\[
\left\|u^\top\{D^2f(\cdot-thu)-D^2f(\cdot)\}u\right\|_2\le 2\|u\|^2\|D^2f\|_{F,2}.
\]
The dominating function
\[
g(u,t):=2\|u\|^2 |K(u)|(1-t)\|D^2 f\|_{F,2},
\]
is integrable over $\mathbb{R}^d\times[0,1]$ by \eqref{eq:multivariate-kernel-moments}. Hence dominated convergence gives
\[
\frac{B_h}{h^2}\to \frac12\mathcal L_K f \qquad \text{in } L^2(\mathbb{R}^d).
\]
Therefore,
\[
\|B_h\|_2^2=\frac{h^4}{4}\|\mathcal L_K f\|_2^2+o(h^4).
\]
The integrated variance satisfies
\[
\int_{\mathbb{R}^d}\Var\{\widehat f_h(x)\}\,dx = \frac{R(K)}{nh^d} + O(n^{-1}),
\]
by the same argument as in Lemma~\ref{lem:integrated-variance}, with $h$ replaced by $h^d$ in the normalization. Since $h\to0$,
\[
n^{-1}=h^d(nh^d)^{-1}= o((nh^d)^{-1}).
\]
Combining the integrated squared bias expansion with the integrated variance expansion proves \eqref{eq:multi-mise-expansion}.
\end{proof}

\begin{corollary}[\textbf{Scalar-bandwidth multivariate AMISE rate}]
\label{cor:multi-amise-rate}
Under the assumptions of Theorem~\ref{thm:multi-amise}, if $h_n\asymp n^{-1/(d+4)}$, then
\[
\MISE(\widehat f_{h_n}) = O\left(n^{-4/(d+4)}\right).
\]
If $\|\mathcal L_K f\|_2>0$, the AMISE-optimal scalar bandwidth is
\begin{equation}
\label{eq:multi-hamise} 
h_{\AMISE}=\left(\frac{dR(K)}{\|\mathcal L_K f\|_2^2 n}\right)^{1/(d+4)}.
\end{equation}
\end{corollary}
\begin{proof}
The scalar-bandwidth AMISE corresponding to Theorem~\ref{thm:multi-amise} is
\[
\AMISE(\widehat f_h)=\frac{h^4}{4}\|\mathcal L_K f\|_2^2+\frac{R(K)}{nh^d}.
\]
If $h_n\asymp n^{-1/(d+4)}$, then
\[
h_n^4=O\left(n^{-4/(d+4)}\right), \qquad (nh_n^d)^{-1}=O\left(n^{-4/(d+4)}\right).
\]
This proves the stated rate. If $\|\mathcal L_K f\|_2>0$, differentiating the AMISE with respect to $h$ gives
\[
h^3\|\mathcal L_K f\|_2^2 -\frac{dR(K)}{nh^{d+1}}= 0.
\]
Hence
\[
h^{d+4} = \frac{dR(K)}{\|\mathcal L_K f\|_2^2 n},
\]
which gives \eqref{eq:multi-hamise}.
\end{proof}

\begin{remark}[\textbf{Full bandwidth matrices}]
The scalar-bandwidth result is included to show how the weak-Hessian argument extends beyond one dimension. Full bandwidth matrices are central in multivariate kernel smoothing, but treating them would require matrix-valued bandwidths and more elaborate curvature functionals. We therefore leave bandwidth-matrix GCPI and anisotropic weak-Hessian AMISE theory for future work.
\end{remark}

\section{Kernel Optimality}
\label{sec:kernel-optimality}

Once the AMISE has the form \eqref{eq:amise-definition}, the usual kernel-optimization argument applies unchanged. The density smoothness affects the factor $R(f'')$, whereas the kernel-dependent part of the optimized AMISE is proportional to
\[
\{\mu_2(K)^2\}^{1/5}R(K)^{4/5} = |\mu_2(K)|^{2/5}R(K)^{4/5}.
\]
Equivalently, one may fix $\mu_2(K)=1$ and minimize $R(K)$ subject to the usual kernel constraints. This gives the classical Epanechnikov-kernel calculation; see \cite{Epanechnikov1969,Silverman1986,WandJones1995}. We therefore state it only as a classical consequence of the weak-curvature AMISE form.

\begin{remark}[\textbf{Classical Epanechnikov consequence}]
\label{rem:epanechnikov}
Among symmetric nonnegative kernels satisfying
\[
\int_{\mathbb{R}}K(u)\,du=1, \qquad \int_{\mathbb{R}}u^2K(u)\,du=1,
\]
the kernel minimizing $R(K)=\int_{\mathbb{R}}K(u)^2\,du$ is the Epanechnikov kernel
\[
    K_E(u)=\begin{cases}
    \displaystyle
    \frac{3}{4\sqrt{5}}\left(1-\frac{u^2}{5}\right),& |u|<\sqrt{5},\\[1ex]
    0, & \text{otherwise.}
    \end{cases}
\]
Thus the classical Epanechnikov kernel continues to minimize the kernel-dependent part of the AMISE once the same AMISE structure has been justified under the weak-curvature formulation.
\end{remark}

\section{Generalized Second-Order Interpretation of the Bias}
\label{sec:generalized-second-order}

The proofs above use weak derivatives and $L^2$ convergence, not generalized Hessians. This is important: the AMISE expansion is an integrated risk statement, and its curvature term is the global weak-curvature functional \eqref{eq:rel_weak_curv}.
Nevertheless, generalized second-order derivatives provide a useful local interpretation of what nonsmooth curvature means for functions in $C^{1,1}$. 

Let $f\in C^{1,1}(\R)$, and let $G$ denote the set of points at which $f$ is twice differentiable. Since $f'$ is locally Lipschitz, Rademacher's theorem implies that $f'$ is differentiable almost everywhere. Thus $G$ has full Lebesgue measure. The generalized second-order derivative of $f$ at $x$ is defined by
\begin{equation}
\label{eq:generalized-second-derivative}
 \partial^2 f(x):=\overline{\co}\left\{\eta\in\mathbb R:\; \exists x_j\in G,\; x_j\to x,\; f''(x_j)\to \eta\right\},
\end{equation}
where $\overline{\co}$ denotes the closed convex hull. This is the one-dimensional version of the Clarke generalized Hessian for $C^{1,1}$ functions; see \cite{Clarke1990,HiriartUrruty1984}.

Because $f'$ is locally Lipschitz, the almost-everywhere derivative $f''$ is locally essentially bounded. Consequently, $\partial^2 f(x)$ is a nonempty compact convex subset of $\mathbb R$. If the classical second derivative exists in a neighborhood of $x$ and is continuous at $x$, then $\partial^2 f(x)=\{f''(x)\}$.
More generally, $\partial^2 f(x)$ records the limiting values of nearby classical second derivatives and then convexifies them. Thus it captures local second-order behavior even when the ordinary second derivative is discontinuous or fails to exist at the point itself.

This generalized object should not be confused with the weak second derivative used in the AMISE calculation. The weak derivative is an almost-everywhere function and is suitable for integration. The generalized derivative $\partial^2 f(x)$ is a local set-valued object. In general, there is no single element $\eta(x)\in\partial^2 f(x)$ that represents the second-order term uniformly inside the KDE kernel integral. The intermediate points in a Taylor-type expansion depend on the integration variable. The weak integral representation used in Lemma~\ref{lem:integral-taylor} avoids this difficulty by averaging the weak second derivative along the relevant line segment.

We now state a simple local bias interpretation. Suppose that the kernel has compact support contained in $[-A,A]$. For $x\in\R$ and $h>0$, define the local generalized curvature bound
\begin{equation}
\label{eq:local-curvature-bound} 
M_h(x) :=\sup\left\{|\eta|:\eta\in\partial^2 f(y),\ |y-x|\le Ah\right\}.
\end{equation}
The local boundedness of $\partial^2 f$ implies that $M_h(x)<\infty$ for sufficiently small $h$. Under the global $C^{1,1}$ assumption used in this paper, one has the simpler uniform bound $M_h(x)\le \operatorname{Lip}(f')$ for all $x\in\mathbb{R}$ and $h>0$.

\begin{proposition}[\textbf{Bias bound by local generalized curvature}]
\label{prop:generalized-curvature-bound}
Let $f\in C^{1,1}(\mathbb{R})$, and let $K$ satisfy Assumption~\ref{ass:kernel} with $\supp K\subset[-A,A]$. Then, for every $x\in\mathbb{R}$ and $h>0$,
\begin{equation}
\label{eq:local-gen-bias-bound} 
\left|\mathbb{E}\widehat f_h(x)-f(x)\right|\le\frac{h^2}{2}M_h(x)\int_{\mathbb{R}}u^2|K(u)|\,du.
\end{equation}
\end{proposition}
\begin{proof}
From the bias representation~\eqref{eq:bias-integral-representation},
\[
\mathbb{E}\widehat f_h(x)-f(x) =h^2\int_{\mathbb{R}}u^2K(u)\int_0^1(1-t)f''(x-thu)\,dt\,du.
\]
Since $\supp K\subset[-A,A]$, for every $u\in\supp K$ and $t\in[0,1]$ we have
\[
|x-thu-x|=th|u|\le Ah.
\]
Thus $x-thu\in[x-Ah,x+Ah]$. At almost every point $y$ in this interval where $f$ is twice differentiable, the value $f''(y)$ is an element of $\partial^2f(y)$. Hence $|f''(y)|\le M_h(x)$ for almost every such $y$. Therefore,
\[
\begin{aligned}
    \left|\mathbb{E}\widehat f_h(x)-f(x)\right|&\le h^2\int_{\mathbb{R}}u^2|K(u)| \int_0^1(1-t)|f''(x-thu)|\,dt\,du       \\ &\le h^2 M_h(x)\int_{\mathbb{R}}u^2|K(u)|\int_0^1(1-t)\,dt\,du \\
    &=\frac{h^2}{2}M_h(x)\int_{\mathbb{R}}u^2|K(u)|\,du.
\end{aligned}
\]
This proves~\eqref{eq:local-gen-bias-bound}.
\end{proof}

\begin{remark}
Proposition~\ref{prop:generalized-curvature-bound} should be read as a local interpretation of the bias, not as an exact asymptotic expansion. It shows that when the kernel is compactly supported, the pointwise bias at $x$ is controlled by the generalized second-order behavior of $f$ in the local window $[x-Ah,x+Ah]$. Exact AMISE constants, however, require the $L^2$ convergence argument used in Theorem~\ref{thm:amise-expansion}.
\end{remark}

\section{Examples of $C^{1,1}\setminus C^2$ Densities}
\label{sec:examples}

The following examples show that the theory applies to densities that are smooth to first order but have nonsmooth curvature. They can arise from threshold mechanisms, regime changes, robust modeling, or compact-support constructions. In all cases, the density has a Lipschitz first derivative and a square-integrable weak second derivative, but it is not twice continuously differentiable.

\begin{example}[\textbf{A Gaussian density with a curvature kink}]
\label{ex:gaussian-kink}
Let
\[
\phi(x)=\frac1{\sqrt{2\pi}}e^{-x^2/2}, \qquad \psi(x)=\frac{x|x|}{1+x^2}.
\]
For $\varepsilon\in(-1,1)$, define
\begin{equation}
\label{eq:kink-density} 
f_\varepsilon(x)=\phi(x)\{1+\varepsilon\psi(x)\}.
\end{equation}
Since $\psi$ is odd and $\phi$ is even,
\[
\int_{\mathbb{R}}\phi(x)\psi(x)\,dx=0.
\]
Also, $|\psi(x)|\le1$, so $f_\varepsilon(x)\ge0$ for $|\varepsilon|<1$. Hence $f_\varepsilon$ is a probability density.

The function $\psi$ is continuously differentiable and has Lipschitz derivative. Indeed,
\[
\psi(x)=
    \begin{cases}
    \displaystyle \frac{x^2}{1+x^2}, & x\ge0,\\[1ex]
    \displaystyle -\frac{x^2}{1+x^2}, & x<0.
    \end{cases}
\]
Thus $\psi'(0)=0$, and $\psi'$ is continuous. However,
\[
\lim_{x\downarrow0}\psi''(x)=2, \qquad \lim_{x\uparrow0}\psi''(x)=-2,
\]
so $\psi$ is not twice differentiable at $0$. It follows that $f_\varepsilon\in C^{1,1}(\mathbb{R})$ but $f_\varepsilon\notin C^2(\mathbb{R})$ whenever $\varepsilon\ne0$.

At the kink point, the generalized second-order derivative is
\[
\partial^2 f_\varepsilon(0) = \phi(0)[-1-2\varepsilon,\,-1+2\varepsilon].
\]
For numerical work one may use the following a.e. formula. For $x\ne0$,
\[
f_\varepsilon''(x) = \phi(x)\Big[(x^2-1)\{1+\varepsilon\psi(x)\} -2\varepsilon x\psi'(x)+\varepsilon\psi''(x)\Big],
\]
where
\[
\psi'(x)=\operatorname{sgn}(x)\frac{2x}{(1+x^2)^2}, \qquad \psi''(x)=\operatorname{sgn}(x)\frac{2(1-3x^2)}{(1+x^2)^3}, \qquad x\ne0.
\]
The value assigned at $x=0$ is immaterial for Lebesgue integration. The weak second derivative exists almost everywhere and belongs to $L^2(\mathbb{R})$, so Theorems~\ref{thm:mise-bound} and~\ref{thm:amise-expansion} apply.
\end{example}

\begin{example}[\textbf{Huber-type density with Gaussian center and exponential tails}]
\label{ex:huber-density}
Let $c>0$ and define the Huber-type potential
\[
U_c(x) =
    \begin{cases}
    \displaystyle \frac{x^2}{2}, & |x|\le c,\\[1ex]
    \displaystyle c|x|-\frac{c^2}{2}, & |x|>c.
    \end{cases}
\]
Let
\[
f_c(x) = \frac{1}{Z_c}\exp\{-U_c(x)\}, \qquad Z_c=\int_{\mathbb{R}}\exp\{-U_c(t)\}\,dt.
\]
Then $f_c$ is a probability density. It has a Gaussian-type core near the origin and exponential tails outside $[-c,c]$. This type of density is natural in robust modeling: the central part behaves like a Gaussian model, while the tails are heavier and less sensitive to outliers.

The potential $U_c$ is continuously differentiable and has Lipschitz derivative. Indeed,
\[
U_c'(x) =
    \begin{cases}
    x, & |x|<c,\\
    c\,\operatorname{sgn}(x), & |x|>c,
    \end{cases}
\]
with matching one-sided limits at $x=\pm c$. However, $U_c''$ jumps from $1$ to $0$ at $x=\pm c$, so $U_c\notin C^2(\mathbb{R})$.

Since $U_c\in C^{1,1}(\mathbb{R})$ and $U_c'$ is bounded, the density $f_c=Z_c^{-1}e^{-U_c}$ belongs to $C^{1,1}(\mathbb{R})$. For $x\ne\pm c$, its second derivative is given by
\[
f_c''(x) = \{U_c'(x)^2-U_c''(x)\}f_c(x),
\]
where
\[
U_c''(x) =
    \begin{cases}
    1, & |x|<c,\\
    0, & |x|>c.
    \end{cases}
\]
At $x=c$, the one-sided limits of $f_c''$ are
\[
\lim_{x\uparrow c}f_c''(x) = (c^2-1)f_c(c), \qquad \lim_{x\downarrow c}f_c''(x) = c^2 f_c(c),
\]
which are different. Hence $f_c\notin C^2(\mathbb{R})$. The same happens at $x=-c$. Moreover, $f_c''\in L^2(\mathbb{R})$ because it is bounded and has exponential tails. Therefore, the weak-curvature AMISE theory applies to $f_c$.
\end{example}

\begin{example}[\textbf{A threshold-regime density}]
\label{ex:threshold-regime-density}
Let $a\in\mathbb{R}$ and $\lambda>0$. Define the piecewise-quadratic potential
\[
U_{a,\lambda}(x) = \frac{x^2}{2} + \frac{\lambda}{2}(x-a)_+^2, \qquad (x-a)_+=\max\{x-a,0\}.
\]
Let
\[
f_{a,\lambda}(x) = \frac{1}{Z_{a,\lambda}}\exp\{-U_{a,\lambda}(x)\}, \qquad Z_{a,\lambda} = \int_{\mathbb{R}}\exp\{-U_{a,\lambda}(t)\}\,dt.
\]
This is a density whose curvature changes after the threshold $a$. Such a form can be viewed as a simple model for a regime change: before the threshold, the potential is quadratic with one curvature, while after the threshold the curvature increases.

The potential $U_{a,\lambda}$ is continuously differentiable and has Lipschitz derivative. Indeed,
\[
U_{a,\lambda}'(x) = x+\lambda(x-a)_+,
\]
which is continuous and Lipschitz. However,
\[
U_{a,\lambda}''(x) =
    \begin{cases}
    1, & x<a,\\
    1+\lambda, & x>a,
    \end{cases}
\]
for $x\ne a$, so $U_{a,\lambda}$ is not twice continuously differentiable at $a$.

The density $f_{a,\lambda}$ belongs to $C^{1,1}(\mathbb{R})$. For $x\ne a$,
\[
f_{a,\lambda}''(x) = \{U_{a,\lambda}'(x)^2-U_{a,\lambda}''(x)\} f_{a,\lambda}(x).
\]
At the threshold $a$, the one-sided limits of $f_{a,\lambda}''$ differ by
$\lambda f_{a,\lambda}(a)$,
so $f_{a,\lambda}\notin C^2(\mathbb{R})$ whenever $\lambda>0$. Since the density has Gaussian-type tails and the a.e. second derivative is square integrable, we have $f_{a,\lambda}''\in L^2(\mathbb{R})$. Thus Theorems~\ref{thm:mise-bound} and~\ref{thm:amise-expansion} apply.
\end{example}

\begin{example}[\textbf{Compactly supported variants}]
\label{ex:compact-supported-kink}
Let $b\in C^\infty_c(\mathbb{R})$ be nonnegative, satisfy $\int_{\mathbb{R}} b=1$, and be strictly positive on an interval $[-2,2]$. Let $\rho\in C^\infty_c(\mathbb{R})$ be supported in $[-1,1]$ and equal to one in a neighborhood of the origin. Define
\[
q(x)=x|x|\rho(x).
\]
Then $q\in C^{1,1}(\mathbb{R})\setminus C^2(\mathbb{R})$, and the nonsmoothness occurs only at the origin. Since $q$ is bounded and supported where $b$ is bounded away from zero, there exists $\varepsilon_0>0$ such that
\[
b(x)+\varepsilon q(x)\ge0,
\]
for all $x$ whenever $|\varepsilon|<\varepsilon_0$. For such $\varepsilon$, the normalized function
\[
f_\varepsilon(x) = \frac{b(x)+\varepsilon q(x)} {\int_{\mathbb{R}}\{b(t)+\varepsilon q(t)\}\,dt},
\]
is a compactly supported density in $C^{1,1}(\mathbb{R})\setminus C^2(\mathbb{R})$. This gives a compactly supported class of examples covered by Theorems~\ref{thm:mise-bound} and~\ref{thm:amise-expansion}.
\end{example}

The examples above illustrate the range of the weak-curvature formulation. The Gaussian perturbation gives a simple analytic test case, the Huber-type density reflects robust modeling with heavier tails, the threshold-regime density represents a change in curvature after a cutoff, and the compactly supported construction shows that the framework is not limited to densities with full support on $\mathbb{R}$. In each case, the density belongs to $C^{1,1}$, fails to be in $C^2$, and has a square-integrable weak second derivative.

\section{Numerical Illustration}\label{sec:numerics}

This section gives a reproducible illustration of the weak-curvature KDE theory. The aim is threefold. First, we show that a density in $C^{1,1}(\mathbb{R})\setminus C^2(\mathbb{R})$ exhibits the usual second-order KDE risk behavior. Second, we verify that the curvature term used in the AMISE formula can be computed from the a.e. weak second derivative. Third, we illustrate the generalized-curvature plug-in (GCPI) bandwidth selector and compare it with the oracle AMISE bandwidth and Silverman's robust normal-reference rule.

The numerical results are intended as a descriptive check of the theory and the computability of GCPI, not as a claim of uniform finite-sample dominance over existing bandwidth selectors. In particular, GCPI depends on a pilot derivative estimator, a pilot bandwidth, and a truncation safeguard, and its finite-sample behavior can vary with these implementation choices.

All experiments were implemented in Python and run on a laptop equipped with a
12th Gen Intel\textsuperscript{\textregistered}
Core\textsuperscript{\texttrademark} i7-12800H CPU at 1.80 GHz and 16 GB of RAM.
The code used to reproduce the simulation results, tables, and figures in this section is
available at \url{https://github.com/elotfian/kde-c11-regularity}.

\subsection{Kinked-curvature simulation}

We used the density \eqref{eq:kink-density} with $\varepsilon=0.5$, displayed in Figure~\ref{fig:kink-density}. Samples were generated by rejection sampling from the standard normal proposal with envelope constant $1+|\varepsilon|$. A proposal draw $Z$ was accepted with probability
\[
\frac{1+\varepsilon\psi(Z)}{1+|\varepsilon|}.
\]
The random seed was set to $123$.

For each Monte Carlo replication, the integrated squared error was approximated by the trapezoidal rule on an equally spaced grid on $[-6,6]$ with 1201 grid points:
\[
\operatorname{ISE} = \int_{-6}^{6} \{\widehat f_h(x)-f_\varepsilon(x)\}^2\,dx.
\]
The weak curvature $R(f_\varepsilon'')$ used in the oracle AMISE bandwidth was computed by numerical quadrature of the analytic a.e. formula for $f_\varepsilon''$ in Example~\ref{ex:gaussian-kink} on $[-12,12]$ with 200001 grid points. This gave
$R(f_\varepsilon'')\approx 0.325427$.
Thus the curvature entering the oracle bandwidth is the weak-curvature functional
\[
R(f_\varepsilon'') = \int_{\mathbb{R}}\{f_\varepsilon''(x)\}^2\,dx,
\]
with $f_\varepsilon''$ interpreted as the weak second derivative.

For each sample size $n$, we averaged over 500 Monte Carlo replications. When comparing kernels, the bandwidth for each kernel was chosen according to the AMISE rule
\[
h_K = \left\{ \frac{R(K)} {\mu_2(K)^2R(f_\varepsilon'')n} \right\}^{1/5}.
\]

\begin{figure}[ht]
\centering
\includegraphics[width=0.72\linewidth]{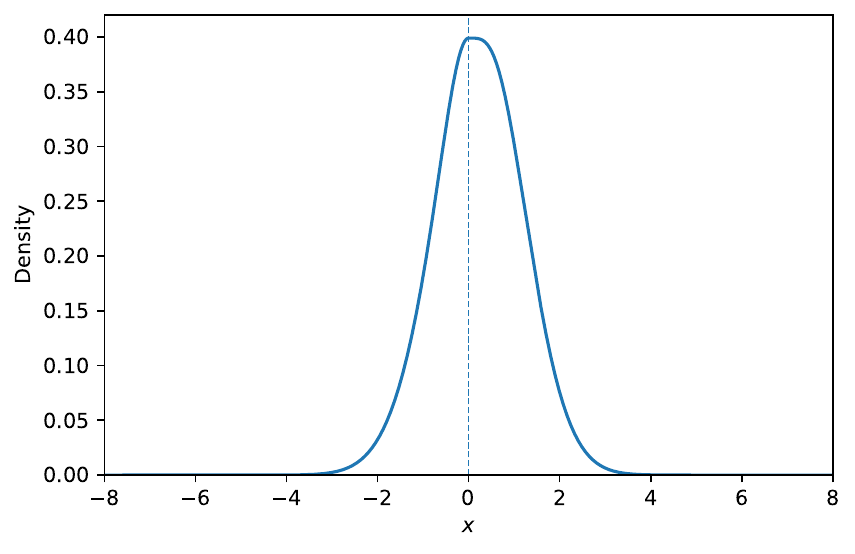}
\caption{A Gaussian density perturbed by an odd curvature-kink term. The density belongs
to $C^{1,1}(\mathbb{R})\setminus C^2(\mathbb{R})$ for $\varepsilon\neq0$, while its weak
second derivative belongs to $L^2(\mathbb{R})$.}
\label{fig:kink-density}
\end{figure}

Table~\ref{tab:simulation} shows the expected qualitative behavior. The mean ISE decreases as $n$ increases for all three kernels, even though the density is not in $C^2$. The finite-sample differences among kernels should not be interpreted as an empirical proof of kernel optimality. The relevant point is more modest: the weak-curvature density behaves consistently with the second-order KDE risk theory.

\begin{table}[ht]
\centering
\caption{Monte Carlo mean integrated squared error for the kinked density
$f_\varepsilon$ with $\varepsilon=0.5$, based on 500 Monte Carlo replications.
Standard errors are shown in parentheses.}
\label{tab:simulation}
\begin{tabular}{llcc}
\toprule
$n$ & Kernel & AMISE bandwidth & Mean ISE \\
\midrule
250  & Epanechnikov & 0.713 & 0.002922 (0.000085) \\
250  & Gaussian     & 0.322 & 0.003050 (0.000087) \\
250  & Biweight     & 0.845 & 0.002934 (0.000086) \\
500  & Epanechnikov & 0.621 & 0.001780 (0.000055) \\
500  & Gaussian     & 0.280 & 0.001858 (0.000056) \\
500  & Biweight     & 0.735 & 0.001787 (0.000055) \\
1000 & Epanechnikov & 0.540 & 0.001050 (0.000028) \\
1000 & Gaussian     & 0.244 & 0.001096 (0.000029) \\
1000 & Biweight     & 0.640 & 0.001055 (0.000028) \\
2000 & Epanechnikov & 0.470 & 0.000606 (0.000015) \\
2000 & Gaussian     & 0.213 & 0.000632 (0.000016) \\
2000 & Biweight     & 0.557 & 0.000608 (0.000015) \\
\bottomrule
\end{tabular}
\end{table}

\subsection{Plug-in bandwidth comparison}

To illustrate the weak-curvature plug-in rule, we compared three Epanechnikov KDE bandwidths: the oracle AMISE bandwidth, the GCPI bandwidth \eqref{eq:gcpi-bandwidth-v5}, and Silverman's robust normal-reference rule. The pilot curvature estimator used the Gaussian pilot kernel.
In the implementation, the diagonal-corrected integrated-square form from Lemma~\ref{lem:diagonal-decomposition} was evaluated by numerical quadrature on the same grid. In the continuum, and up to numerical quadrature error, this corresponds to the leave-one-out U-statistic correction.
We used the pilot bandwidth $b_n=n^{-1/6}$, satisfying $0<1/6<2/9$ as required by Remark~\ref{rem:pilot-assumptions}, and the truncation safeguard $\tau_n=10^{-8}$.

The resulting mean ISE values, based on 500 Monte Carlo replications for each sample size,
are shown in Table~\ref{tab:plugin-simulation} and displayed on a log--log scale in
Figure~\ref{fig:mise-scaling}. This comparison is intended to illustrate finite-sample behavior and computability of the weak-curvature selector; it is not a formal dominance claim. In particular, the finite-sample behavior of GCPI can depend on the pilot bandwidth and the truncation safeguard, and these tuning choices are not systematically optimized here.

\begin{table}[ht]
\centering
\caption{Illustrative comparison of Epanechnikov KDE bandwidth selectors for the kinked
density. The GCPI selector estimates the weak-curvature functional $R(f'')$ using a pilot
curvature estimator. Monte Carlo standard errors are shown in parentheses.}
\label{tab:plugin-simulation}
\resizebox{\linewidth}{!}{%
\begin{tabular}{lcccc}
\toprule
$n$ & Oracle ISE & GCPI ISE & Silverman ISE
& Median $\widehat h_{\operatorname{GCPI}}/h_{\AMISE}$ \\
\midrule
250  & 0.002922 (0.000085) & 0.003186 (0.000095) & 0.007338 (0.000147) & 1.17 \\
500  & 0.001780 (0.000055) & 0.001876 (0.000059) & 0.004365 (0.000089) & 1.14 \\
1000 & 0.001050 (0.000028) & 0.001095 (0.000030) & 0.002460 (0.000044) & 1.12 \\
2000 & 0.000606 (0.000015) & 0.000625 (0.000016) & 0.001420 (0.000023) & 1.11 \\
\bottomrule
\end{tabular}%
}
\end{table}
\begin{figure}[ht]
\centering
\includegraphics[width=0.72\linewidth]{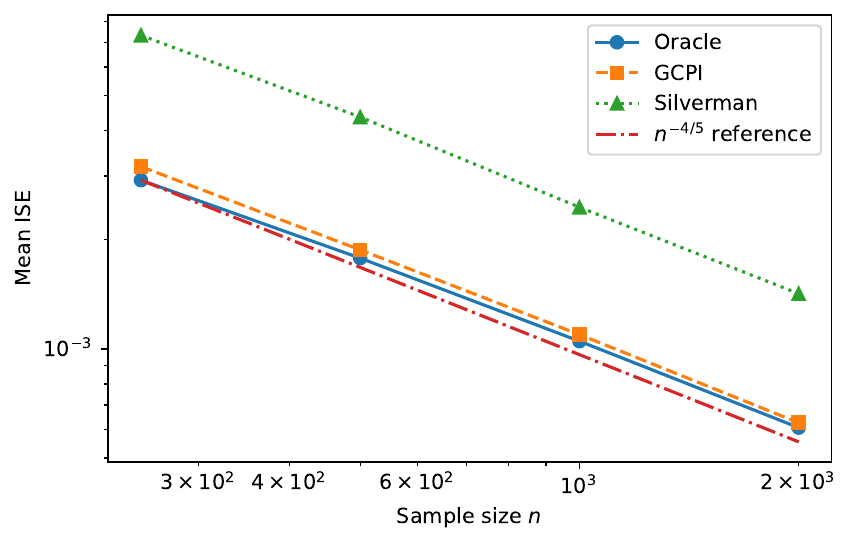}
\caption{Log--log plot of Monte Carlo mean integrated squared error against sample size
for the kinked-curvature density. The reference line has slope $-4/5$, corresponding to
the theoretical second-order KDE rate.}
\label{fig:mise-scaling}
\end{figure}

Table~\ref{tab:plugin-simulation} and Figure~\ref{fig:mise-scaling} are consistent with the theoretical message. The oracle and GCPI curves follow the $n^{-4/5}$ reference trend closely, while Silverman's robust normal-reference rule has a larger finite-sample ISE for this non-Gaussian kinked-curvature density. The GCPI bandwidth is somewhat larger than the oracle AMISE bandwidth in these samples: the median ratio $\widehat h_{\operatorname{GCPI}}/h_{\AMISE}$ decreases from $1.17$ at $n=250$ to $1.11$ at $n=2000$. This mild oversmoothing is consistent with the fact that the raw GCPI curvature estimates are below the true weak-curvature value on average in this run. The plot should therefore be read as a descriptive check that the weak-curvature plug-in selector has risk decay compatible with the second-order AMISE rate, not as a claim of finite-sample optimality.

\subsection{Old Faithful eruption-duration data}

As a small real-data illustration, we applied three bandwidth selectors to the eruption-duration variable in the classical Old Faithful geyser data. The R documentation describes the \texttt{faithful} data as 272 observations on two variables: eruption time in minutes and waiting time to the next 
eruption\footnote{faithful: Old Faithful Geyser Data}. We used only the eruption durations, since they are a standard one-dimensional example for KDE and have a visibly non-Gaussian shape.

For this example, all displayed density estimates use the Epanechnikov kernel. We compare Silverman's robust normal-reference bandwidth, least-squares cross-validation, and the GCPI bandwidth. For the GCPI calculation, the Gaussian pilot kernel was used with pilot bandwidth
\[
b=s_{\mathrm{rob}} n^{-1/9}, \qquad s_{\mathrm{rob}}=\min\{s,\operatorname{IQR}/1.34\},
\]
where $s$ is the sample standard deviation and $\operatorname{IQR}$ is the sample interquartile range. Since the true density is unknown, Table~\ref{tab:faithful-bandwidths} reports the selected bandwidth and the number of visible modes of the resulting density estimate rather than an ISE value. The number of modes is used only as a descriptive diagnostic and depends on the grid and merging tolerance used to summarize the displayed curves.

\begin{table}[ht]
\centering
\caption{Bandwidth selectors for the Old Faithful eruption-duration data. The number of
modes is counted from the displayed Epanechnikov KDE on a fine grid and is intended only as
a descriptive diagnostic.}
\label{tab:faithful-bandwidths}
\begin{tabularx}{\linewidth}{@{}>{\raggedright\arraybackslash}X
                            c
                            c
                            >{\raggedright\arraybackslash}X@{}}
\toprule
Selector & Bandwidth & Modes & Interpretation \\
\midrule
Silverman robust normal-reference&0.335&3&Smooth scale-based baseline with a small shoulder
\\
\addlinespace
Least-squares cross-validation&0.192&8&More variable estimate with several small local modes
\\
\addlinespace
GCPI&0.623&2&Weak-curvature plug-in estimate preserving the main bimodality
\\
\bottomrule
\end{tabularx}
\end{table}

Figure~\ref{fig:faithful-bandwidths} shows the corresponding density estimates. The main qualitative feature of the eruption-duration data is bimodality. The cross-validation bandwidth produces several small local modes, while the GCPI estimate keeps the two dominant modes and avoids fine-scale oscillations. This example does not prove superiority of the proposed selector. Its role is more modest: it shows that the weak-curvature plug-in rule can be computed on a standard real dataset and can produce an interpretable density estimate without relying on a Gaussian reference shape.

\begin{figure}[ht]
\centering
\includegraphics[width=0.9\linewidth]{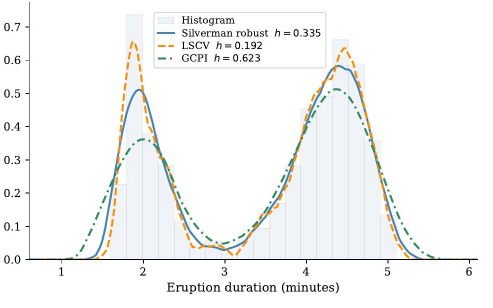}
\caption{Epanechnikov kernel density estimates for the Old Faithful eruption-duration data
using three bandwidth selectors. The GCPI selector estimates the weak-curvature functional,
whereas the normal-reference rule uses a scale approximation and least-squares
cross-validation directly optimizes a sample version of integrated risk.}
\label{fig:faithful-bandwidths}
\end{figure}

\section{Discussion}\label{sec:discussion}

The classical KDE analysis based on a $C^2$ Taylor expansion is mathematically convenient, but stronger than necessary for the second-order AMISE calculation. For the pointwise $O(h^2)$ bias bound, Lipschitz continuity of the first derivative is sufficient. For the sharp integrated AMISE constant, the relevant curvature condition is the square-integrability of the weak second derivative, together with $f\in L^2(\mathbb R)$. Thus the usual AMISE expansion remains valid with the curvature functional $R(f'')$ interpreted in the weak sense, while $C^{1,1}\setminus C^2$ provides a concrete and interpretable nonsmooth class covered by the theory. 

The generalized second-order derivative offers a useful nonsmooth-analysis interpretation of local curvature, especially at points where classical second differentiability fails. However, for integrated risk analysis, weak derivatives and $L^2$ translation continuity provide a cleaner route to exact AMISE constants. This distinction is important: generalized Hessian elements are local set-valued objects, whereas AMISE is an integrated risk criterion depending on the global weak-curvature functional.

Several extensions remain natural. The GCPI result developed here is first-order and proof-of-concept: it does not optimize the pilot bandwidth, derive a limiting distribution for the curvature estimator, or establish adaptive minimax optimality. Future work may develop full bandwidth-matrix and anisotropic versions of the weak-Hessian AMISE expansion, local or adaptive bandwidth rules, and boundary-corrected versions for compactly supported densities whose curvature has nonsmooth features near the boundary. A more complete bandwidth-selection theory would also require sharper analysis of the pilot curvature estimator, including rates, limiting distributions, and fully data-driven pilot selection.

\section*{Statements and Declarations}

\paragraph{Funding}
A. Kabgani was supported by the Research Foundation--Flanders (FWO) through research project G081222N and by the University of Antwerp Research Fund (BOF) through DOCPRO4 project 46929.

\paragraph{Competing interests}
The authors declare that they have no competing interests.

\paragraph{Data availability}
The Old Faithful eruption-duration data used in the real-data illustration are publicly available through the standard \texttt{faithful} dataset in R.

\paragraph{Code availability}
The Python code used to reproduce the simulation results, tables, and figures is available at \url{https://github.com/elotfian/kde-c11-regularity}.

\paragraph{Author contributions}
Both authors contributed to the conception, analysis, writing, and revision of the manuscript.

\end{document}